\documentclass[12pt]{article}
\usepackage{times}
\usepackage{amsmath,amsfonts,amstext,amssymb,amsbsy,amsopn,amsthm,eucal,slashed,amstext,enumerate,accents}

\usepackage{color}
\usepackage{txfonts}
\usepackage{dsfont}
\usepackage{hyperref}
\usepackage{graphicx}   
\usepackage{comment}
\usepackage{scalerel,stackengine}
\stackMath
\newcommand\reallywidehat[1]{%
\savestack{\tmpbox}{\stretchto{%
  \scaleto{%
    \scalerel*[\widthof{\ensuremath{#1}}]{\kern-.6pt\bigwedge\kern-.6pt}%
    {\rule[-\textheight/2]{1ex}{\textheight}}
  }{\textheight}%
}{0.5ex}}%
\stackon[1pt]{#1}{\tmpbox}%
}
\numberwithin{equation}{section}
  \theoremstyle{plain}
 \newtheorem{theorem}[equation]{Theorem}
 
\newtheorem{proposition}[equation]{Proposition}
 \newtheorem{lemma}[equation]{Lemma}
  \newtheorem{claim}[equation]{Claim}
 \newtheorem{corollary}[equation]{Corollary}

 \theoremstyle{remark}

 \newtheorem{remark}[equation]{Remark}

\theoremstyle{definition}
 \newtheorem{definition}[equation]{Definition}

\newtheorem{example}[equation]{Example}

\newcommand{\abs}[1]{\lvert#1\rvert}
\newcommand{\norm}[1]{\lvert\lvert#1\rvert\rvert}
\newcommand{\ip}[1]{\left\langle#1\right\rangle}

\newcommand{\Rc}{{\rm Ric}}
\newcommand{\dd}{{\, d}}

\newcommand{\dE}{\mathds{E}}

\newcommand{\dR}{\mathds{R}}

\newcommand{\eps}{\varepsilon}
\newcommand{\cA}{\mathcal{A}}

\newcommand{\cH}{\mathcal{H}}

\renewcommand{\P}{\mathbb{P}}

\newcommand{\cR}{\mathcal{R}}

\newcommand{\tr}{\text{tr}}
\newcommand{\Lap}{\Delta}

\newcommand{\RR}{\mathbb{R}}
\newcommand{\PP}{\mathbb P}
\newcommand{\EE}{\mathbb E}

\newcommand{\Eo}{\mathbb E_{0}}
\newcommand{\Ex}{\mathbb E_{x}}

\DeclareMathOperator{\Hess}{Hess}
\DeclareMathOperator{\Ric}{Ric}
\DeclareMathOperator{\Rm}{Rm}

\begin{document}

\title{Differential Harnack Inequalities on Path Space}
\author{Robert Haslhofer, Eva Kopfer, Aaron Naber \thanks{The first author has been supported by NSERC grant RGPIN-2016-04331 and a Sloan Research Fellowship. The second author has been supported by the German Research Foundation through the Hausdorff Center for Mathematics and the Collaborative Research Center 1060. The third author has been supported by NSF grant DMS-1809011. All three authors thank the Fields Institute in Toronto for support during the thematic program on Geometric Analysis.}}
\date{\today}

\maketitle

\begin{abstract}
Recall that if $(M^n,g)$ satisfies $\Ric\geq 0$, then the Li-Yau Differential Harnack Inequality tells us for each nonnegative $f:M\to \dR^+$, with $f_t$ its heat flow, that $\frac{\Delta f_t}{f_t}-\frac{|\nabla f_t|^2}{f_t^2} +\frac{n}{2t}\geq 0.$  Our main result will be to generalize this to path space $P_xM$ of the manifold.\\

A key point is that instead of considering infinite dimensional gradients and Laplacians on $P_xM$ we will consider, in a spirit similar to \cite{Naber_char,haslhofernaberricci}, a family of finite dimensional gradients and Laplace operators.  Namely, for each $H^1_0$-function $\varphi:\dR^+\to \dR$ we will define the $\varphi$-gradient $\nabla_\varphi F: P_xM\to T_xM$ and the $\varphi$-Laplacian $\Delta_\varphi F =\text{tr}_\varphi\Hess F:P_xM\to \dR$, where $\Hess F$ is the Markovian Hessian and both the gradient and the $\varphi$-trace are induced by $n$ vector fields naturally associated to $\varphi$ under stochastic parallel translation.\\

Now let $(M^n,g)$ satisfy $\Ric=0$, then for each nonnegative $F:P_xM\to \dR^+$ we will show the inequality $$\frac{\Ex [\Delta_\varphi F]}{\Ex [F]}-\frac{\Ex [\nabla_\varphi  F]^2}{\Ex [F]^2} +\frac{n}{2}|| \varphi ||^2\geq 0$$ for each $\varphi$, where $\Ex$ denotes the expectation with respect to the Wiener measure on $P_xM$. By applying this to the simplest functions on path space, namely cylinder functions of one variable $F(\gamma) \equiv f(\gamma(t))$, we will see we recover the classical Li-Yau Harnack inequality exactly.  We have similar estimates for Einstein manifolds, with errors depending only on the Einstein constant, as well as for general manifolds, with errors depending on the curvature.  Finally, we derive generalizations of Hamilton's Matrix Harnack inequality on path space $P_xM$.  It is our understanding that these estimates are new even on the path space of $\dR^n$.
\end{abstract}

\newpage

\tableofcontents

\newpage

\section{Introduction}

\subsubsection*{Differential Harnack Inequalities on Manifolds}

Let us open by recalling the classical differential Harnack inequalities on manifolds.  Thus, consider a Riemannian manifold $(M^n,g)$ and for $f:M\to \dR$ denote by $ f_t=H_tf :M\to \dR$ the solution of the heat equation $(\partial_t - \Delta)f_t=0$ with $f_0=f$.  The classical Li-Yau differential Harnack inequality \cite{LiYau} tells us that if $f$ is nonnegative and if $\Ric\geq 0$, then we have
\begin{align}\label{e:li_yau_harnack}
\frac{\Delta f_t}{f_t}-\frac{|\nabla f_t|^2}{f_t^2} +\frac{n}{2t}\geq 0\, .	
\end{align}
While there are many other useful sharp estimates on heat flows which play an important role in analysis for manifolds with nonnegative Ricci curvature, for instance the Bakry-Emery \cite{BakryEmery} estimate $|\nabla H_t f|\leq H_t|\nabla f|$, the differential Harnack inequality distinguishes itself in that it directly incorporates the dimension into the underlying estimate.  Thus, the differential Harnack inequality is the usual starting point for many estimates on heat kernels, and other estimates which directly rely on the underlying dimension.  For instance, integrating along a suitable space-time geodesic gives the sharp classical Harnack estimate
\begin{equation}
f_{t_2}(x_2)\geq \left(\tfrac{t_1}{t_2}\right)^{n/2}e^{-\frac{d(x_1,x_2)^2}{4(t_2-t_1)}} f_{t_1}(x_1)\, .
\end{equation}
The differential Harnack inequality \eqref{e:li_yau_harnack} and many of its implications are sharp and obtained when considering the heat kernel on Euclidean space.  As another application we can apply \eqref{e:li_yau_harnack} to the heat kernel $\rho_{x,t}(y)=\rho_t(x,y)$, centered at some point $x\in M$, in order to obtain the estimate
\begin{align}\label{e:li_yau_heat_kernel}
\Delta \ln\rho_{x,t} \geq -\frac{n}{2t}\, .
\end{align}
One can interpret the above as a smoothing of the classical Laplacian comparison theorems for the distance function.  In addition to the Li-Yau Harnack inequality there is also Hamilton's Matrix Harnack inequality \cite{Hamilton_matrix_harnack}.  In the context where one assumes the stronger geometric constraints $\nabla\Ric=0$ and $\sec\geq 0$, Hamilton proved the Hessian version of \eqref{e:li_yau_harnack} given by
\begin{align}\label{e:matrix_harnack}
\frac{\nabla^2 f_t}{f_t}-\frac{\nabla f_t\otimes \nabla f_t}{f_t^2} +\frac{g}{2t}\geq 0\, .	
\end{align}

\subsubsection*{Harnack and Basics of Path Space $P_xM$}

The goal of this paper is to extend the differential Harnack inequalities to the context of the path space $P_xM$ of a manifold.  We will have generalizations of the Li-Yau differential Harnack inequality \eqref{e:li_yau_harnack}, the Hamilton Matrix Harnack inequality \eqref{e:matrix_harnack}, and the heat kernel estimate \eqref{e:li_yau_heat_kernel} to the path space context.  These extensions will require some work to detail, which we will do step by step over the next several subsection, for now let us open with some general comments followed by some standard constructions on analysis on path space.  To begin, let us be careful and remark that our notion of path space is the collection of continuous based paths:
\begin{align}
P_xM \equiv \big\{\gamma\in C^0([0,\infty),M):\gamma(0)=x\big\}\, .	
\end{align}
Performing analysis on $P_xM$, like performing analysis on any space, involves three important ingredients: A nice dense collection of functions to work with, a measure to integrate with, and a notion of gradient.  The first two of these ingredients will be standard notions in this context, which we will review now.  The notion of gradient we will introduce in this paper, and its induced Laplacian, will be new.  The $\varphi$-gradient $\nabla_\varphi$ and $\varphi$-Laplacian $\Delta_\varphi$ will act more as a family of finite dimensional gradients and Laplacians, in the spirit of \cite{Naber_char,haslhofernaberricci}.  We will introduce these a little more slowly over the coming subsections.\\

Let us now finish our introductory review by dealing with the first two ingredients above, namely the construction of nice functions and the Wiener measure.  Both are built using the canonical evaluation maps on path space.  Namely, consider a partition ${\bf{t}} =\{0<t_1<\cdots<t_k<\infty\}$, then from this we can build the evaluation map $e_{{\bf{t}}}:P_xM\to M^k$ given by
\begin{equation}
e_{\bf{t}}(\gamma) = (\gamma_{t_1},\ldots,\gamma_{t_k})\, .
\end{equation}
From this we can generate functions on $P_xM$ by pullback.  Namely, given a partition ${\bf{t}}$ and a function $f:M^k\to \dR$ the induced cylinder function $F:P_xM\to \dR$ on path space is given by
\begin{equation}
F(\gamma) = e_{\bf{t}}^*f(\gamma) = f(\gamma_{t_1},\ldots,\gamma_{t_k})\, .
\end{equation}
These functions have a distinctly finite dimensional quality to them, and as such will be particularly easy and natural to work with.  In the end these functions will be dense in every space of functions we need to work on, so it will be sufficient to do most computations with respect to them.  

 In a similar vein, path space $P_xM$ is equipped with a natural probability measure $\PP_x$, called the Wiener measure, which is uniquely defined through its pushforwards by the evaluation maps:
\begin{align}\label{e:Wiener_measure_pushforwards}
e_{{\bf{t}}\,\ast}\PP_x = \rho_{t_1}(x,dx_1)\rho_{t_2-t_1}(x_1,dx_2)\cdots \rho_{t_k-t_{k-1}}(x_{k-1},dx_k)\, ,	
\end{align}
where $\rho_t(x,dy)=\rho_t(x,y)dv_g(y)$ are the heat kernel measures.  It is a beautiful classical result that $\PP_x$ exists as a measure on continuous path space $P_xM$.  In this way the Wiener measure not only tells us about the heat kernels at all times and points, but also how they interact with one another.\\

Let us now move ourselves toward the new results, during which time we will introduce the notions of gradient and Laplacian that will prove themselves most important.

\subsection{Differential Harnack Inequalities on Path Space of $\dR^n$}

Let us begin by analyzing the context of path space on flat Euclidean space.  Our results are new even in this setting, and beyond that it will be an excuse to analyze the estimates and inequalities in a context where many of the technical bells and whistles will not be present.  We will be interested in studying continuous paths based at the origin:
\begin{align}
P_0\dR^n = \big\{\gamma\in C^0\big([0,\infty),\dR^n\big):\gamma(0)=0\big\}\, .	
\end{align}
On $P_0\dR^n$ we can consider the Wiener probability measure $\PP_0$, defined as in \eqref{e:Wiener_measure_pushforwards}.  It is interesting to observe one can view this measure as a Gaussian measure on $P_0\dR^n$ with standard deviation coming from the $H^1$-norm.  As such, when performing analysis on path space it is convenient to often restrict ourselves to directions which are $H^1$ in nature, which gives rise to the Cameron-Martin space:
\begin{equation}\label{e:cameron_martin}
\mathcal{H}=\left\{ h\in P_0(\mathbb{R}^n) \, : \, \norm{h}^2_{\cH} \equiv \int_0^\infty \abs{\dot h}^2dt <\infty\right\}\, .
\end{equation}

  Our first main result in the rigid context of path space on $\dR^n$ is the following, which we will use as an inspiration for our generalized Matrix Harnack inequality in the path space setting:\\

\begin{theorem}[convexity]\label{thm_harnack_path_space}
If $F:P_0\mathbb{R}^n\to \mathbb{R}^+$ is a positive integrable function, then the associated functional
\begin{equation}\label{eq_def_phi}
\Phi_F: \mathcal{H}\to\mathbb{R},\quad \Phi_F(h)=\ln\left(\int_{P_0\mathbb{R}^n} F(\gamma+h)\, d\PP_0(\gamma)\right) + \frac{1}{4}\norm{h}^2_{\cH}
\end{equation}
is convex.
\end{theorem}

We will provide the short proof of the above in Section \ref{sec_euclidean}, for now let us consider an enlightening example obtained by applying the above to the simplest functions on path space:

\begin{example}\label{example_intro}
Consider the cylinder function $F:P_0\dR^n\to \dR^+$	 given by $F(\gamma) = f(\gamma(t))$, where $f:\dR^n\to \dR^+$ and $t>0$ are fixed.  Consider a linear curve $h(s)\equiv \frac{s}{t}x\in \dR^n$ connecting the origin to a point $x\in \dR^n$, and for each direction $v\in \dR^n$ and each $r\in\mathbb{R}$ consider the perturbation of $h$ in the $v$ direction given by $h_r(s)\equiv \frac{s}{t}(x+rv)\in \dR^n$ for $s\leq t$.  That is, $h_r(s)$ is simply the straight curve from the origin to $h_r(t)=x+rv$, so in particular $h_0(s)=h(s)$.  Now using the pushforward characterization \eqref{e:Wiener_measure_pushforwards} of the Wiener measure  we can compute
\begin{align}
\Phi_F(h_r) = \ln\left(\int_{\dR^n} f(y+x+rv)\rho_t(0,dy) \right)+\frac{|x+rv|^2}{4t}	= \ln f_t(x+rv)+\frac{|x+rv|^2}{4t}\, .
\end{align}
Then the convexity condition $\frac{d^2}{dr^2}\Big|_{r=0}\Phi_F(h_r)\geq 0$ converts to the inequality
\begin{align}
\frac{\nabla^2 f_t(v,v)}{f_t} - \frac{\ip{ \nabla f_t, v}^2}{f_t^2} +\frac{|v|^2}{2t}\geq 0
\end{align}
for every $v$, which is precisely the Matrix Harnack inequality \eqref{e:matrix_harnack}. $\qed$\\
\end{example}

Generalizing the above example, given $f:\mathbb{R}^{n\times k}\to\mathbb{R}^+$ and $0<t_1<\ldots < t_k$ we can consider 
\begin{equation}
H_{t_1,\ldots,t_k}f(x_1,\ldots,x_k):=\int_{\mathbb{R}^{n k}} f(y_1+x_1,\ldots,y_k+x_k) \rho_{t_1}(0,dy_1)\rho_{t_2-t_1}(y_1,dy_2)\ldots \rho_{t_k-t_{k-1}}(y_{k-1},dy_k),
\end{equation}
which can be interpreted as a (completely correlated) generalization of the heat flow, and obtain:

\begin{corollary}[convexity for generalized heat flow]
The function $\psi_f: \mathbb{R}^{n\times k}\to \mathbb{R}$,
\begin{equation}
\psi_f (x_1,\ldots,x_k)=\ln H_{t_1,\ldots,t_k}f(x_1,\ldots,x_k)+\frac{|x_1|^2}{4t_1}+\frac{|x_2-x_1|^2}{4(t_2-t_1)}+\ldots+\frac{|x_k-x_{k-1}|^2}{4(t_k-t_{k-1})}
\end{equation}
is convex.
\end{corollary}

We have therefore seen that Theorem \ref{thm_harnack_path_space} behaves as a natural path space generalization of the Matrix Harnack Inequality, and indeed recovers it exactly when applied to the simplest functions on path space.  

Our next challenge is that Theorem \ref{thm_harnack_path_space} as written does not generalize to manifolds.  We will therefore look for weak reformulations which have some hope of being defined on general manifolds.  This will eventually lead us to our differential Harnack inequalities.  

There are many approaches one can naturally take to write Theorem \ref{thm_harnack_path_space} weakly, the statements and definitions of our next results are motivated by giving a presentation which will extend in a natural manner to more general manifolds.  We begin by introducing the $\varphi$-gradient in the Euclidean context:

\begin{definition}\label{def_eucl_phi_grad}
Let $\varphi:[0,\infty)\to \dR$ be an $H^1_0$-function, i.e.  $||\varphi||^2\equiv \int |\dot\varphi|^2 <\infty $ and $\varphi(0)=0$.  For $F:P_0\dR^n\to \dR$ we define its $\varphi$-gradient $\nabla_{\varphi} F:P_0\dR^n\to \dR^n$ by
\begin{align}
\ip{ \nabla_{\varphi} F(\gamma), v} \equiv D_{\varphi v} F =\lim_{\eps\to 0} \frac{F(\gamma+\eps\varphi v)-F(\gamma)}{\eps}\, .	
\end{align}
\end{definition}

Note that, in a spirit similar to \cite{Naber_char,haslhofernaberricci}, the $\varphi$-gradient $\nabla_{\varphi} F$ is essentially a finite dimensional gradient as it only considers information about the derivative of $F$ in those directions determined by $\varphi$.  By considering an orthonormal basis $\{\varphi_j\}$ we can recover the full Malliavin gradient $\nabla^{\mathcal{H}} F:P_0\dR^n\to \cH$.  

In addition to the $\varphi$-gradient we will want to define the associated $\varphi$-Hessians and $\varphi$-Laplacians:

\begin{definition}\label{def_eucl_phi_hess}
Let $\varphi:[0,\infty)\to \dR$ be an $H^1_0$-function.  For $F:P_0\dR^n\to \dR$ we define 
\begin{enumerate}
\item The $\varphi$-Hessian $\Hess_{\varphi} F:P_0\dR^n\to \dR^{n\times n}$ given by\footnote{It is worth observing that the definition in the Euclidean context is greatly simplified, as $\varphi v$ is a constant vector field and thus $D_{\varphi v} D_{\varphi w}F$ is a Hessian.  In the general case we must subtract off the correct Christoffel symbol.}
\begin{align}
\ip{ \Hess_{\varphi} F(\gamma), v\otimes w} \equiv D_{\varphi v} D_{\varphi w}F \, .
\end{align}	
\item The $\varphi$-Laplacian $\Delta_\varphi F:P_0\dR^n\to \dR$ given by $\Delta_{\varphi}\, F= \tr(\Hess_{\varphi} F)$.
\end{enumerate}
\end{definition}

Considering an orthonormal basis $\{\varphi_j\}$ we can recover the $H^1$-Laplacian $\Delta_{\cH}$, so that in this way we have naturally decomposed the infinite dimensional Laplacian into a sum of finite dimensional Laplacians.  We can now use Theorem \ref{thm_harnack_path_space} in order to prove the following:

\begin{theorem}[Differential Harnack inequality on path space of Euclidean space]\label{thm_diff_harnack_eucl}
If $F:P_0\dR^n\to \dR^+$ is a nonnegative integrable function, then for all test functions $\varphi\in H^1_0(\dR^+)$ we have 
\begin{align}
 &\frac{\dE\left[\Hess_{\varphi}F\right]}{\dE[F]}
 -\frac{\dE\left[\nabla_\varphi F\right]\otimes \dE\left[\nabla_\varphi F\right]}{\dE[F]^2}
 +\frac{1}2\norm{\varphi}^2\geq 0\, ,
\end{align}
where $\dE$ denotes the expectation with respect to the Wiener measure $\PP_0$.  In particular, we can trace to obtain
\begin{align}
 &\frac{\dE\left[\Delta_{\varphi}F\right]}{\dE[F]}
 -\frac{\left|\dE\left[\nabla_\varphi F\right]\right|^2}{\dE[F]^2}
 +\frac{n}{2}\norm{\varphi}^2\geq 0.
\end{align}
\end{theorem}

Theorem \ref{thm_diff_harnack_eucl} can be viewed as an infinite family of finite dimensional differential Harnack inequalities on path space. It is not hard to see that Theorem \ref{thm_diff_harnack_eucl} and Theorem \ref{thm_harnack_path_space} in fact imply each other. The formulation as a differential Harnack inequality, as opposed to a convexity statement, is more suitable for our generalizations to the path space of manifolds.\\

\subsection{The $\varphi$-Gradient and $\varphi$-Laplacian}

In order to state our results on general manifolds we need to discuss the notion of $\varphi$-gradients and $\varphi$-Laplacians on manifolds.  Let us begin by defining the notion of the $\varphi$-gradient:

\begin{definition}[$\varphi$-gradient]\label{d:phi_gradient}
Let $F:P_xM\to \dR$ be a cylinder function, and let $\varphi:[0,\infty)\to \dR$ be an $H^1_0$-function, i.e. a function such that $||\varphi||^2\equiv \int |\dot\varphi|^2<\infty $ and $\varphi(0)=0$.   Then we define the $\varphi$-gradient $\nabla_\varphi F:P_xM\to T_xM$ by 
\begin{align}
	\ip{\nabla_\varphi F,v} = D_{\varphi V} F,
\end{align}
where $V$ is the vector field along $\gamma$ obtained by parallel translating $v$ along $\gamma$,\footnote{We need to use the stochastic parallel translation map to make this precise on a generic curve, see Section \ref{sec_prelim}.} and thus $D_{\varphi V}$ is the directional derivative of $F$ in the direction $\varphi V\in T_\gamma P_xM$.
\end{definition}

The $\varphi$-gradient is essentially a finite dimensional gradient, in a spirit similar to \cite{Naber_char,haslhofernaberricci}.  It contains information about the directional derivatives of $F$ in all directions determined by $\varphi$.  As in the Euclidean case, by considering an orthonormal basis $\{\varphi_j\}$ of $H^1_0$ we see that we can recover the full Malliavin-gradient $\nabla^{\mathcal{H}} F:P_xM \to \cH$. \\

In order to define a Hessian we must consider covariant derivatives of vector fields on path space.  Two considerations when defining a connection on $P_xM$ are that one wishes it to be compatible with the $H^1_0$-metric, and wishes it to preserve adapted vector fields.  Among such connections there is a best choice, which was introduced in Cruzeiro-Malliavin \cite{cruzeiromalliavin}, called the Markovian connection.  
To define the Markovian connection, recall that vector fields $V$ on $P_xM$ can be identified with functions $v_t:P_xM\to T_xM\equiv \mathbb{R}^n$ via parallel transport. Namely, we can take $V(\gamma)_t\in T_{\gamma_t}M$ and map it using the parallel translation map $P_t(\gamma): T_{\gamma_t}M\to T_x M$ to get
\begin{equation}
v_t(\gamma):= P_t(\gamma)V(\gamma)_t\in T_xM.
\end{equation}

\begin{definition}[Markovian Connection]\label{d:intro:mark_conn}
	The Markovian connection $\nabla$ on $P_xM$ is given by\footnote{To be precise, the integral should be viewed as Stratonovich integral, see Section \ref{sec_prelim}.}
\begin{align}\label{e:d:intro:mark_conn}
\frac{d}{dt}P_t(\nabla_V W)_t = D_V\dot{w}_t+ \left(\int_0^tP_{s}\Rm_{\gamma_s}(V_s,\dot{\gamma}_s)\, ds\right) \, \dot{w}_t\, ,	
\end{align}
where $P_t:T_{\gamma_t}M\to T_xM$ denotes the parallel translation map, and where $w_t=P_tW_t$.
\end{definition}

We note that the curvature term in \eqref{e:d:intro:mark_conn} arises as the derivative of the parallel translation map.

 Given the Markovian connection $\nabla$, the Markovian Hessian of a function $F:P_xM\to \mathbb{R}$ is now naturally defined by
 \begin{align}\label{d:intro:mark_hess}
 \Hess F (V,W) \equiv  D_V(D_W F) - D_{\nabla_V W} F\, ,
 \end{align}
where $D$ denotes the directional derivatives. Using this, we can now introduce the $\varphi$-Hessian and $\varphi$-Laplacian, which will play a central roles in our differential Harnack inequalities:

\begin{definition}[$\varphi$-Hessian and $\varphi$-Laplacian]\label{d:phi_laplacian}
Let $F:P_xM\to \dR$ be a cylinder function, and let $\varphi:[0,\infty)\to \dR$ be an $H^1_0$-function, i.e. a function such that $||\varphi||^2\equiv \int |\dot\varphi|^2<\infty $ and $\varphi(0)=0$. 
\begin{enumerate}
\item We define $\Hess_\varphi F:P_xM\to T_x^\ast M\otimes T_x^\ast M$ by
\begin{align}
	\Hess_\varphi F(v,v) = \Hess F(\varphi V,\varphi V)\, ,
\end{align}
where $V$ is the vector field along $\gamma$ obtained by parallel translating $v$ along $\gamma$.
\item $\Delta_\varphi F = \tr\big(\Hess_\varphi F\big):P_xM\to \dR$ is the $\varphi$-Laplacian obtained by tracing the $\varphi$-Hessian. 
\end{enumerate}
\end{definition}

To understand the meaning of this definition, consider for each $\varphi$ an $n$-dimensional distributional $E_\varphi\subset T P_xM$ given by
\begin{align}\label{e:varphi_distribution}
E_\varphi = \text{span}\big\{\,\varphi V: \text{$V$ is the parallel translation of a vector $v\in T_xM$}\big\}	\, .
\end{align}
Thus, at each $\gamma\in P_xM$ we have that $E_\varphi(\gamma)$ is an $n$-dimensional subspace of $T_\gamma P_xM$.  Then the $\varphi$-Hessian and the $\varphi$-Laplacian are simply given by
\begin{equation}
\Hess_\varphi F = \Hess F \big|_{E_\varphi\otimes E_\varphi }\, ,
\end{equation}
and
\begin{equation}
\Delta_\varphi F = \tr_{E_\varphi} \Hess F\,.
\end{equation}

In particular, the $\varphi$-Laplacian is simply the trace of the infinite dimensional Hessian along the finite dimensional subspace $E_\varphi$.  Hence, in the same spirit as the $\varphi$-gradients, the $\varphi$-Laplacians behave as a family of finite dimensional Laplacians.  This is crucial for us, as our generalization of the differential Li-Yau Harnack inequality will actually be a family of inequalities, one for each $\Delta_\varphi$.\\

\subsection{Differential Harnack Inequalities on Path Space of Ricci Flat Manifolds}

Now we are in a position to discuss our first more general estimates.  We begin with the Ricci-flat context primarily because the estimates are cleaner and easier to digest.  The general cases will follow in the next subsections.  Our main theorem in the Ricci-flat case is the following:\\

\begin{theorem}[Differential Harnack inequality on path space]\label{thmintro:harnack_ricci_flat}
Let $M$ be a Ricci-flat manifold, and let $F:P_xM\to \dR$ be a nonnegative function. Then, for all $\varphi\in H^1_0(\mathbb{R}^+)$ we have the inequality
\begin{align}\label{e:harnack_ricci_flat_intro}
 &\frac{\Ex\left[\Delta_{\varphi}F\right]}{\Ex[F]}
 -\frac{\big|{\Ex\left[\nabla_\varphi F\right]}\big|^2}{\Ex[F]^2}
 +\frac{n}2\norm{\varphi}^2\geq 0.
\end{align}
\end{theorem}

Let us begin, as we often like to, by applying this to the simplest functions on path space in order to see that we can recover the classical Li-Yau Harnack inequality:\\

\begin{example}[Li-Yau inequality]\label{ex_intro_LiYau}
Let us consider the cylinder function $F:P_xM\to \dR^+$	 given by $F(\gamma) = f(\gamma(t))$, where $f:M\to \dR^+$ and $t>0$ are fixed. Let $\varphi:[0,\infty)\to \dR$ be such that $\varphi(s) = \frac{s}{t}$ for $s\leq t$ and $\varphi(s)=1$ for $s\geq t$.  One can use the definition of the $\varphi$-gradient to immediately compute
\begin{align}
&\nabla_\varphi F(\gamma) = P_t(\gamma)\nabla f(\gamma(t))\, ,
\end{align}
where $P_t(\gamma):T_{\gamma(t)}M\to T_x M$ denotes parallel transport.  Now let $e_i\in T_xM$ be an orthonormal basis with $E_i$ the associated parallel translation invariant vector fields along each $\gamma$. Using the definition of the Markovian connection (Definition \ref{d:intro:mark_conn}) we see that
\begin{equation}
\sum_{i=1}^n\nabla_{\varphi E_i}\varphi E_i = 0,
\end{equation}
where the curvature term disappeared after taking the trace since $\Ric=0$. It follows that
\begin{align}
\Delta_\varphi F(\gamma) = \Delta f(\gamma(t)).
\end{align}
Using the above and the Feynman-Kac formula we can then derive the equalities
\begin{align}
&\Ex[F] = \int_M f(y)\rho_t(x,dy) = f_t(x)\, ,\notag\\
&\Ex[\Delta_\varphi F] = \Delta f_t(x)\, ,\notag\\
&\Ex[\nabla_\varphi F]  = \nabla f_t(x)\, ,
\end{align}
where in the last equality we used again that $\Ric=0$. Finally, observing that $||\varphi||^2 = \frac{1}{t}$ and plugging all of this into \eqref{e:harnack_ricci_flat_intro} we obtain
\begin{align}
	\frac{\Delta f_t}{f_t} - \frac{|\nabla f_t|^2}{f_t^2} + \frac{n}{2t} \geq 0\, ,
\end{align}
which is precisely the Li-Yau Harnack inequality. $\qed$
\end{example}

Another consequence is a generalization of the Li-Yau estimate \eqref{e:li_yau_heat_kernel} on heat kernels:

\begin{example}[Laplacian of the log of the Wiener Measure]\label{ex_laplogW} By plugging in a smoothed Dirac delta function into Theorem \ref{thmintro:harnack_ricci_flat} we formally obtain the Laplace comparison estimate\footnote{Here, we view the energy functional $-\ln\PP_x	\equiv \frac{1}{4}\int |\dot\gamma|^2$ as the log of the Wiener measure, motivated by the integration by parts formula. 
Note that although $\ln\PP_x$ is not defined on continuous path space, its gradient is.}
\begin{equation}\label{e:laplace_comparison_Wiener}
\Delta_\varphi \ln \PP_x \geq -\frac{n}{2}\, ,
\end{equation}
for each $\varphi$ with $||\varphi||=1$. To interpret this, recall from \eqref{e:varphi_distribution} that for each $\varphi$ we have an associated $n$-dimensional distribution $E_\varphi$ on $TP_xM$, and that $\Delta_\varphi = \tr_{E_\varphi}\Hess$. Thus, the estimate \eqref{e:laplace_comparison_Wiener} is telling us that the trace of the Hessian of $\ln\PP_x$ is bounded below on each of the $n$-dimensional subspaces $E_\varphi$.  Hence, $\ln\PP_x$ behaves like a plurisubharmonic function on a complex manifold.  
\end{example}

\begin{remark}[Equality]
Computing more carefully one can check that actually equality is attained in the above example. Namely, the log of the Wiener measure satisfies the interesting identity
\begin{equation}
\Delta_\varphi \ln \PP_x =- \frac{n}{2}
\end{equation}
for each normalized $\varphi$. We emphasize that this only holds if $M$ is Ricci-flat.
\end{remark}

Alternatively, instead of in terms of the $H^1$-geometry, our differential Harnack inequality on path space of Ricci-flat manifolds can also be understood in terms of the $L^2$-geometry of path space. To this end, we denote by $\Hess_\varphi^\mathcal{L}$ and $\Delta_\varphi^\mathcal{L}$ the $\varphi$-Hessian and $\varphi$-Laplacian that are obtained by using the $L^2$-connection $\nabla^\mathcal{L}$ instead of the Markovian connection $\nabla$. Concretely, we have
\begin{align}
	\Hess_\varphi^\mathcal{L} F(v,v) = \frac{d^2}{ds^2}\Big|_{s=0}F(\gamma_s)\, ,
\end{align}
where $\gamma_s$ is a family of curves with $\partial_s|_{s=0}\gamma_s = \varphi V$ and $\nabla_{\varphi V}\big(\partial_s\gamma_s\big) = 0$,\footnote{For instance $\gamma_s(t) = \exp_{\gamma(t)}(s\varphi(t)V(t))$ gives such a curve.}
and
\begin{equation}
\Delta_\varphi^\mathcal{L} F = \tr\Hess_\varphi^\mathcal{L} F\,  .
\end{equation}

\begin{corollary}[Differential Harnack inequality in terms of $L^2$-geometry]\label{corintro:harnack_ricci_flat}
Let $M$ be a Ricci-flat manifold, and let $F:P_xM\to \dR$ be a nonnegative function. Then, for all $\varphi\in H^1_0(\mathbb{R}^+)$ we have the inequality
\begin{align}\label{e:harnack_ricci_flat_intro2}
 &\frac{\Ex\left[\Delta^\mathcal{L}_{\varphi}F\right]}{\Ex[F]}
 -\frac{\big|{\Ex\left[\nabla_\varphi F\right]}\big|^2}{\Ex[F]^2}
 +\frac{n}2\norm{\varphi}^2\geq 0.
\end{align}
\end{corollary}

In fact, we will show in Section \ref{sec_harnack_L2} that on path space of Ricci-flat manifolds, the $\varphi$-Laplacian induced by the $L^2$-connection agrees with the one induced by the Markovian connection.

\bigskip

\subsection{Differential Harnack Inequalities on Path Space of General Manifolds}

The situation for general manifolds is quite analogous to the previous section, though unsurprisingly we now get more error terms depending on the curvatures.
Our main differential Harnack inequality on the path space of general manifolds is the following:
 
\begin{theorem}[Differential Harnack inequality on path space]\label{theorem_harn_gen}
Let $F:P_xM\to\RR^+$ be a nonnegative $\Sigma_T$-measurable function on path space. Then, for every $\varphi\in H_0^1(\mathbb{R}^+)$ we have the inequality
 \begin{align}\label{e:diff_harn_general}
\frac{\Ex[\Delta_{\varphi} F]}{\Ex[F]}
-\frac{\abs{\Ex[\nabla_\varphi F]}^2}{\Ex[F]^2}
+\left( \frac{n}{2}+C_T(\Ric)
+C_T(\Rm,\nabla\Ric)\frac{\Ex[{ F}^2]^{1/2}}{\Ex[F]}  \right) 
\norm{\varphi}^2\geq 0,
\end{align}
where $C_T(\Ric)<\infty$ and $C_T(\Rm,\nabla\Ric)<\infty$ are constants, which converge to $0$ as $\abs{\Ric}+|\nabla\Ric|\to0$ assuming that $\abs{\Rm}$ and $T$ stay bounded.
\end{theorem}

Theorem \ref{theorem_harn_gen} generalizes Theorem \ref{thmintro:harnack_ricci_flat} to the path space of general manifolds. Again, it provides an infinite dimensional family of finite dimensional differential Harnack inequalities on path space $P_xM$.
There are a couple points about the error terms worth observing.  They depend on the $L^2$-norm of $F$. In general, they further depend on bounds on the full curvature tensor $|\Rm|$ and on $|\nabla\Ric|$.  This seems to be a feature of second order estimates on path space, in contrast to the first order estimates of \cite{Naber_char,haslhofernaberricci}, where the errors only depend on the Ricci curvature, and nothing involving the full curvature or the covariant derivative of curvature. However, if the  underlying manifold is Einstein, then as a corollary of our proof we obtain:  
 
 \begin{corollary}\label{cor_matr_harn}
If $M$ is Einstein, i.e. $\Ric=\Lambda g$, then the constants only depend on $\Lambda$, namely
 \begin{align}
\frac{\Ex[\Delta_{\varphi} F]}{\Ex[F]}
-\frac{\abs{\Ex[\nabla_\varphi F]}^2}{\Ex[F]^2}
+\left( \frac{n}{2}+C_T(\Lambda)\left(1+\frac{\Ex[{ F}^2]^{1/2} }{\Ex[F]} \right) \right) 
\norm{\varphi}^2\geq 0,
\end{align}
where $C_T(\Lambda)\to 0$ as $\Lambda\to0$ assuming that $T$ stays bounded.
\end{corollary}

\begin{remark}
We saw in the Ricci-flat case that $\Delta_\varphi$ may be replaced by $\Delta_\varphi^\mathcal{L}$.  However, this is absolutely not the case in general, even if $M$ is Einstein.  The difference between the Markovian and $L^2$ quantities involves terms that are fundamentally not controllable in the form of \eqref{e:diff_harn_general}.
\end{remark}

\bigskip

\subsection{Differential Matrix Harnack Inequalities on Path Space}

Finally, we discuss our differential Matrix Harnack inequality on path space, meant to generalize Hamilton's Matrix Harnack Inequality \eqref{e:matrix_harnack}:

\begin{theorem}[Differential Matrix Harnack inequality on path space]\label{thmintro: general}
Let $F:P_xM\to\RR^+$ be a nonnegative $\Sigma_T$-measurable function on path space. 
Then, for every $\varphi\in H_0^1(\mathbb{R}^+)$ we have the inequality
 \begin{multline}
\frac{\Ex[\Hess_{\varphi} F]}{\Ex[F]}
-\frac{\Ex[\nabla_\varphi F]\otimes \Ex[\nabla_\varphi F] }{\Ex[F]^2}\\
+ \left(\frac{1}{2}
+C_T(\Ric)
+C_T(\Rm,\nabla\Ric)\frac{\Ex[ {F}^2]^{1/2} }{\Ex[F]}\right) \norm{\varphi}^2g_x\geq 0,
\end{multline}
where $C_T(\Ric)<\infty$ and $C_T(\Rm,\nabla\Ric)<\infty$ are constants, which converge to $0$ as $\abs{\Rm}+|\nabla\Ric|\to0$ assuming that $T$ stays bounded.
\end{theorem}

In the path space context one only gets a full errorless estimate in the flat case.  That is, similar to Hamilton's Matrix Harnack inequality, which assumes $\sec\geq 0$ and $\nabla\Ric=0$, even Ricci-flatness is not enough to obtain Hessian estimates without error terms.  This should not be surprising, as the full Hessian estimates inevitably involve estimates on parallel translation maps, which involve the full curvature tensor.  Compared to the manifold case, Theorem \ref{thmintro: general} again contains completely new global information capturing the interaction between different points.\\

\subsection{Other Generalizations}

The differential Harnack inequalities of the previous sections were in terms of the $\varphi$-Hessian and $\varphi$-Laplacian, which themselves depended on a choice of connection on $P_xM$.  Our chosen connection on $P_xM$, namely the Markovian connection $\nabla$, is the one that is popular in the literature, however the differential Harnack inequalities do in fact hold for a wide class of connections on $P_xM$.  The Markovian condition \eqref{e:d:intro:mark_conn} can be generalized to the condition
\begin{align}\label{e:generalizations_connections:1}
\frac{d}{dt}P_t(\nabla^\mathcal{A}_V W)_t = D_V\dot{w}_t+\mathcal{A}_t(\gamma,V)\, \dot{w}_t\, ,
\end{align}
where $\mathcal{A}_t(\gamma,V):T_xM\to T_xM$.  Then so long as for each bounded $V$ we have that $\mathcal{A}_t$ is an adapted process which is also an $L^2$ antisymmetric mapping, then the induced connection is an $H^1$-connection which preserves adapted vector fields for which the Harnack inequalties of this paper hold.  Let us consider two important scenarios, beginning with the Cartan connection on $P_xM$:\\

\begin{definition}[Cartan Connection]\label{d:intro:cartan_conn}
The Cartan connection $\nabla^\mathcal{C}$ on $P_xM$ is the unique connection such that vector fields of the form $\varphi V$ are parallel, where $\varphi$ is an $H^1_0$-function and  $V$ is the vector field on $P_xM$ obtained by parallel translating a fixed $v\in T_xM$ along each $\gamma$.  
\end{definition}

The Cartan connection satisfies $\frac{d}{dt}P_t(\nabla^\mathcal{C}_V W)_t = D_V\dot{w}_t$, and thus \eqref{e:generalizations_connections:1} holds with $\cA\equiv 0$.
The Cartan connection $\nabla^\mathcal{C}$ is a flat connection on $P_xM$ which is not torsion free, indeed its torsion is now related to the curvature of $M$ itself.  In particular, one can prove the {\it verbatim} differential Harnack inequalities stated in this paper hold with the Hessian and Laplacian induced by this connection as well.\\

Finally, let us consider a non-example.  Another interesting choice of connection on $P_xM$ is the $L^2$-connection $\nabla^\mathcal{L}$.  Indeed, on Ricci-flat spaces the $\varphi$-Laplacians induced by the Markovian connection, the Cartan connection, and $L^2$-connection are all the same.  However, the $L^2$-connection in the form of \eqref{e:generalizations_connections:1} looks like $\frac{d}{dt}P_t(\nabla^\mathcal{L}_V W)_t = D_V\dot{w}_t + \Rm_{\gamma_t}(V_s,\dot\gamma_t) w_t+ \left(\int_0^tP_{s}\Rm_{\gamma_s}(V_s,\dot{\gamma}_s)\, ds\right) \, \dot{w}_t\,$.  The additional curvature term $\Rm_{\gamma_t}(V_s,\dot\gamma_t)$ is clearly not an $L^2$ function on $P_xM$.  The effect of this is that in non-Ricci flat case (or indeed for the Matrix Harnack even in the Ricci-flat case) the differential Harnack inequalities of this paper do not hold.  One obtains new errors (see the anticipating integral in Proposition \ref{prop: mark hess}) which fundamentally cannot be controlled in the same fashion.\\

\bigskip

\subsection{Outline of the Paper}

Let us briefly outline the paper along with the main steps of the proof.

In Section \ref{sec_euclidean}, we give the proof of our Harnack estimate Theorem \ref{thm_harnack_path_space} on the path space $P_0\dR^n$ of Euclidean space.  The proof in this context comes down to nothing more than a computation involving the Cameron-Martin change of variables formula and H\"older's inequality.  Regardless, this simple setting allows for a good starting point for developing intuition.

In Section \ref{sec_prelim}, we discuss the required preliminaries regarding stochastic analysis on manifolds.  After recalling the Wiener measure and the stochastic parallel translation map, we will spend some time discussing the different notions of gradients which appear in this paper.  These notions, and in particular the gradients of vector fields, can give rise to some subtle points on the path space analysis.  This is in part because there are several different such notions, each meant to capture different behaviors.  Finally, at the end of Section \ref{sec_prelim} we will discuss the intertwining and integration by parts formula.  We will state and prove the integration by parts formula for continuous adapted processes, which is a somewhat more general form than the most popular one.  This form of the integration by parts formula will be needed in future steps. 

In Section \ref{sec: ricflat}, we will give the proofs of our Harnack results in the Ricci-flat context.  The proofs in the Ricci-flat case will be very similar to the general case of Section \ref{sec: general}, however we can avoid many technicalities which can otherwise bog one down.  The first main result in Section \ref{sec: ricflat} is the Halfway Harnack of Theorem \ref{thmRicflat: quadratic form} which shows that the quadratic form
\begin{align}
Q_F[V,V]:=\frac{\Ex[D_{V}(D_{V} F)]}{\Ex[F]}-\frac{\Ex[D_{V} F]^2}{\Ex[F]^2}+\frac{\Ex[D_{\nabla^{}_{V}V}F]}{\Ex[F]}+\frac12\norm{V}_\cH^2\geq 0,
\end{align}
is nonnegative for all adapted vector fields $V$ on path space $P_xM$. Here, $\nabla$ is the Markovian connection from Cruzeiro-Malliavin \cite{cruzeiromalliavin}, see Definition \ref{d:intro:mark_conn}.  One can view this Halfway Harnack as a nongeometric version of the Harnack inequality, as fundamentally one can view it as the pushforward of our Harnack on Eulidean path space under the Ito map.  This Halfway Harnack of a function $F$ itself is then only half the picture, as we need to remove the non-tensorial terms, as well as estimate a variety of a-priori arbitrary looking curvature terms hidden inside the definition of the Markovian connection.   When combined with the correct tracing formulas this will allow us to turn the Halfway Harnack into the full differential Harnack.

In Section \ref{sec: general}, we end by generalizing the differential Harnack to the path space of arbitrary manifolds. Indeed, this is very similar spirit to the Ricci-flat context, however everything is a good deal more technical.  In particular, we will see it is important to use a twisted notion of gradient, which will interact better with the methods of this paper in the non Ricci-flat context.

\bigskip

\section{The Euclidean Case}\label{sec_euclidean}

In this short section, as warmup for the later sections, we prove our differential Harnack inequalities in the simple setting of path space of $\mathbb{R}^n$. We start by establishing convexity of the functional $\Phi_F$ from
\eqref{eq_def_phi}.

\begin{proof}[{Proof of Theorem \ref{thm_harnack_path_space}}]
Let $h_1,h_2\in\mathcal{H}$ and $\lambda_1,\lambda_2\in(0,1)$ with $\lambda_1+\lambda_2=1$. We have to show that
\begin{equation}
\Phi_F(\lambda_1 h_1+\lambda_2 h_2)\leq \lambda_1 \Phi_F(h_1)+ \lambda_2 \Phi_F(h_2).
\end{equation}
To this end, note that by the Cameron-Martin theorem \cite{CameronMartin} we have the change of variables formula
\begin{equation}\label{eq_cam_mart}
\int_{P_0\mathbb{R}^n} F(\gamma+h) \, d\mathbb{P}_0(\gamma)=\int_{P_0\mathbb{R}^n} F(\gamma)e^{\frac{1}{2}\ip{ h,\gamma}-\frac{1}{4}\norm{h}^2} \, d\mathbb{P}_0(\gamma),
\end{equation}
where $\ip{ h,\gamma} = \int_0^\infty \dot{h}_t\,dW_t(\gamma)$ is given as Ito integral of the process $\dot{h}_t$ with respect to Brownian motion. Using this, a short computation yields
\begin{align}
\Phi_F(\lambda_1 h_1+\lambda_2 h_2)&=\ln\left(\int_{P_0\mathbb{R}^n} F(\gamma)e^{\frac{1}{2}\ip{ \lambda_1 h_1+ \lambda_2 h_2,\gamma} } d\mathbb{P}_0(\gamma)\right) \nonumber\\
&\leq \lambda_1 \ln\left(\int_{P_0\mathbb{R}^n} F(\gamma)e^{\frac{1}{2}\ip{ h_1,\gamma} } d\mathbb{P}_0(\gamma)\right) + \lambda_2 \ln\left(\int_{P_0\mathbb{R}^n} F(\gamma)e^{\frac{1}{2}\ip{ h_2,\gamma} } d\mathbb{P}_0(\gamma)\right)\nonumber \\
&= \lambda_1 \Phi_F(h_1)+ \lambda_2 \Phi_F(h_2),
\end{align}
where we used the change of variables formula \eqref{eq_cam_mart} in the first and third line, and H\"older's inequality in the second line. This proves the theorem.
\end{proof}

Considering the most simple functions and variations on path space, Theorem \ref{thm_harnack_path_space} implies Hamilton's Matrix Harnack inequality \eqref{e:matrix_harnack} as explained in Example \ref{example_intro}. More generally, we obtain the following corollary.

\begin{corollary}\label{cor_convexity}
If $k$ is a positive integer, $f:\mathbb{R}^{n\times k}\to\mathbb{R}_+$ is a positive function (say of subexponential growth), and $0< t_1< \ldots <t_k$, then the associated function $\psi^f_{t_1,\ldots,t_k}:\mathbb{R}^{n\times k}\to\mathbb{R}$,
\begin{multline}\label{eq_conv_k}
(x_1,\ldots,x_k)\mapsto \ln\left(\int_{\mathbb{R}^{n\times k}} f(y_1+x_1,\ldots,y_k+x_k) \rho_{t_1}(0,dy_1)\rho_{t_2-t_1}(y_1,dy_2)\ldots \rho_{t_k-t_{k-1}}(y_{k-1},dy_k)  \right) \\
+ \frac{|x_1|^2}{4t_1}+ \frac{|x_2-x_1|^2}{4(t_2-t_1)}+ \ldots +  \frac{|x_k-x_{k-1}|^2}{4(t_k-t_{k-1})}
\end{multline}
is convex. In particular, for $k=1$ we see that the function
\begin{equation}
x\mapsto \ln\left(\int_{\mathbb{R}^{n}} f(z) \frac{e^{-\frac{\abs{x-z}^2}{4t}}}{(4\pi t)^{n/2}}dz \right)+ \frac{|x|^2}{4t}
\end{equation}
is convex, which by computing the Hessian reduces to Hamilton's Matrix Harnack inequality  \eqref{e:matrix_harnack}.
\end{corollary}

\begin{remark}\label{rem1_after_cor}
Generalizing the intuition of the case $k=1$, it is useful to interpret the functions
\begin{equation}
H_{t_1,\ldots,t_k}f(x_1,\ldots,x_k):=\int_{\mathbb{R}^{n k}} f(y_1+x_1,\ldots,y_k+x_k)\rho_{t_1}(0,dy_1)\ldots \rho_{t_k-t_{k-1}}(y_{k-1},dy_k)
\end{equation}
appearing in \eqref{eq_conv_k} as a generalized heat flow for $k$-point functions. A particularly interesting feature of Corollary \ref{cor_convexity} is that it gives also information about the mixed Hessians $\nabla_{x_i}\nabla_{x_j} H_{t_1,\ldots,t_k}f(x_1,\ldots,x_k)$ for $i\neq j$.
\end{remark}

\begin{remark}\label{rem2_after_cor} Another useful way to understand the generalized heat flow for $k$-point functions is to rewrite it quite redundantly as an integral over $k$ Brownian motions in $\mathbb{R}^n$:
\begin{equation}\label{eq_comp_corr}
H_{t_1,\ldots,t_k}f(x_1,\ldots,x_k)=\int_{P_0(\mathbb{R}^n)^k} f(\gamma_{t_1}+x_1,\ldots ,\gamma_{t_k}+x_k)\chi_{\{\gamma^1=\ldots = \gamma^k\}} d\mathbb{P}_0(\gamma^1)\ldots d\mathbb{P}_0(\gamma^k).
\end{equation}
The indicator function $\chi_{\{\gamma^1=\ldots = \gamma^k\}}$ enforces that these $k$ Brownian motions are completely correlated, i.e. they are actually all the same. The formula \eqref{eq_comp_corr} can be compared with the opposite extreme, the completely uncorrelated case, which is obtained by dropping the indicator function, namely
\begin{equation}
U_{t_1,\ldots,t_k}f(x_1,\ldots,x_k):=\int_{P_0(\mathbb{R}^n)^k} f(\gamma_{t_1}+x_1,\ldots ,\gamma_{t_k}+x_k) d\mathbb{P}_0(\gamma^1)\ldots d\mathbb{P}_0(\gamma^k).
\end{equation}
In particular, in the special case that $f(y_1,\ldots, y_k)= f_1(y_1)\cdots f_k(y_k)$ is a product function, this simply becomes a product of heat flows, namely
\begin{equation}
U_{t_1,\ldots,t_k}({f_1\cdots f_k})(x_1,\ldots,x_k)= H_{t_1}f_1(x_1)\cdots H_{t_k}f_k(x_k).
\end{equation}
And simply adding up the Hamilton's Matrix Harnack expressions for these $k$ heat flows one sees that
\begin{equation}
(x_1,\ldots,x_k)\to \ln U_{t_1,\ldots,t_k}({f_1\cdots f_k})(x_1,\ldots,x_k) + \tfrac{\abs{x_1}^2}{4t_1}+\ldots+\tfrac{\abs{x_k}^2}{4t_k}
\end{equation}
is also convex. Of course, the convexity in the completely correlated case is the much more interesting one, and the one that doesn't simply follow by applying Hamilton's Matrix Harnack inequality $k$ times, but for the sake of intuition it is quite useful to keep in mind these two opposite extreme cases.
\end{remark}

\begin{proof}[{Proof of Corollary \ref{cor_convexity}}]
Given the function $f:\mathbb{R}^{n\times k}\to\mathbb{R}_+$ and the times $0< t_1< \ldots <t_k$ we can define a positive function on path space by setting
\begin{equation}
F(\gamma):=f(\gamma_{t_1},\ldots,\gamma_{t_k})\, .
\end{equation}
Now thinking of the times $0< t_1< \ldots <t_k$ as fixed, to any $k$-points $x_1,\ldots,x_k\in\mathbb{R}^n$ we associate a Cameron-Martin vector $h^{x_1,\ldots,x_k}\in\mathcal{H}$ by defining
\begin{equation}
h^{x_1,\ldots,x_k}_t:=\left\{\begin{array}{ll}
        \tfrac{t}{t_1}x_1 & \text{for } 0\leq t\leq t_1\\
        x_1+\tfrac{t-t_1}{t_2-t_1}(x_2-x_1) & \text{for } t_1\leq t\leq t_2\\
        \ldots & \ldots \\
        x_{k-1}+\tfrac{t-t_{k-1}}{t_k-t_{k-1}}(x_{k}-x_{k-1}) & \text{for } t_{k-1}\leq t\leq t_k\\
        x_k & \text{for } t \geq t_k
        \end{array}\right.
\end{equation}
Since $h^{x_1,\ldots,x_k}_t$ is piecewise linear it is easy to compute that
\begin{equation}
\norm{h}^2= \frac{|x_1|^2}{t_1}+ \frac{|x_2-x_1|^2}{(t_2-t_1)}+ \ldots +  \frac{|x_k-x_{k-1}|^2}{(t_k-t_{k-1})}.
\end{equation}
To proceed, we recall that if $e_{t_1,\ldots,t_k}:P_0(\mathbb{R}^n)\to \mathbb{R}^{nk}, \gamma\mapsto (\gamma_{t_1},\ldots, \gamma_{t_k})$ denotes the evaluation map at the times $0< t_1< \ldots<t_k$ then the pushforward of the Wiener measure is given by a product of heat kernel measures, namely
\begin{equation}
e_{t_1,\ldots,t_k \,\ast} \mathbb{P}_0=  \rho_{t_1}(0,dy_1)\rho_{t_2-t_1}(y_1,dy_2)\ldots \rho_{t_k-t_{k-1}}(y_{k-1},dy_k).
\end{equation}
Using this, we compute
\begin{multline}
\int_{P_0\mathbb{R}^n} F(\gamma+h) \, d\mathbb{P}_0(\gamma)
=\int_{P_0\mathbb{R}^n} f(\gamma_{t_1}+h_{t_1},\ldots,\gamma_{t_k}+h_{t_k})\,  d\mathbb{P}_0(\gamma)\\
=\int_{\mathbb{R}^{n\times k}} f(y_1+x_1,\ldots,y_1+x_1)\,  \rho_{t_1}(0,dy_1)\rho_{t_2-t_1}(y_1,dy_2)\ldots \rho_{t_k-t_{k-1}}(y_{k-1},dy_k).
\end{multline}
Now let us define $\ell_{t_1,\ldots,t_k}:\mathbb{R}^{n\times k}\to\mathcal{H}, (x_1,\ldots,x_k)\mapsto h^{x_1,\ldots,x_k}_t$ and observe that this is a linear map. Since by Theorem \ref{thm_harnack_path_space} the functional $\Phi_F: \mathcal{H}\to \mathbb{R}$ is convex, the composed function $\Phi_F\circ \ell_{t_1,\ldots,t_k}:\mathbb{R}^{n\times k}\to \mathbb{R}$ is also convex. The above computation shows that $\psi^f_{t_1,\ldots,t_k}=\Phi_F\circ \ell_{t_1,\ldots,t_k}$, and this proves the corollary.
\end{proof}

To conclude this section, let us prove our differential Harnack inequalities on path space of $\mathbb{R}^n$:

\begin{proof}[{Proof of Theorem \ref{thm_diff_harnack_eucl}}]
Given any vector $v\in\RR^n$ and any function $\varphi\in H^1_0(\mathbb{R}^+)$ we consider the direction
\begin{equation}
h_\eps(t) = \eps \varphi(t)v.
\end{equation}
As a consequence of Theorem \ref{thm_harnack_path_space}, the function
\begin{equation}
\eps\mapsto \Phi_F(h_\eps)=\ln \EE\left[ F(\gamma+ \eps \varphi v) \right] + \frac{1}{4}|| \eps \varphi v||^2
\end{equation}
is convex. Now, a straightforward computation yields
\begin{equation}
\frac{d}{d\eps} \Phi_F(h_\eps)=\frac{\EE\left[ D_{\varphi v}F(\gamma+ \eps \varphi v) \right]}{\EE\left[ F(\gamma+ \eps \varphi v) \right]} + \frac{1}{2} \eps |v|^2  ||\varphi ||^2,
\end{equation}
and
\begin{equation}
\frac{d^2}{d\eps^2}|_{\eps=0} \Phi_F(h_\eps)=\frac{\EE\left[ D_{\varphi v}D_{\varphi v}F \right]}{\EE\left[ F \right]} - \left( \frac{\EE\left[ D_{\varphi v}F \right] }{\EE\left[ F \right]}\right)^2 + \frac{1}{2} |v|^2  ||\varphi ||^2\geq 0.
\end{equation}
Recalling the definitions of the $\varphi$-gradient (Definition \ref{def_eucl_phi_grad}) and $\varphi$-Hessian (Definition \ref{def_eucl_phi_hess}), we conclude that
\begin{align}
 &\frac{\dE\left[\Hess_{\varphi}F\right]}{\dE[F]}
 -\frac{\dE\left[\nabla_\varphi F\right]\otimes \dE\left[\nabla_\varphi F\right]}{\dE[F]^2}
 +\frac{\delta}2\norm{\varphi}^2\geq 0\, .
\end{align}
This proves the theorem.
\end{proof}

\begin{remark}\label{rem_quad_form_rn}
As a motivation for the analysis in the manifold case, let us record that the second variation in a general direction $h\in\cH$ is given by
\begin{align}\label{quad_form_eucl}
\nabla^2\Phi_F(0)[h,h]=\frac{\EE\left[ D_{h}(D_{h }F) \right]}{\EE\left[ F \right]} - \left( \frac{\EE\left[ D_{h}F \right] }{\EE\left[ F \right]}\right)^2 + \frac{1}{2} ||h||^2.
\end{align}
\end{remark}

\bigskip

\section{Preliminaries for the Manifold Case}\label{sec_prelim}

In this section we briefly discuss some preliminaries regarding the analysis on path space. Standard references for stochastic analysis on manifolds are the books by Hsu \cite{Hsu} and Stroock \cite{Stroock}.

In the following, $M$ denotes an $n$-dimensional Riemannian manifold (either compact or complete with Ricci curvature bounded below). Given any $x\in M$, recall from the introduction that path space
\begin{align}
 P_xM=\{\gamma\colon[0,\infty) \to M \, | \, \gamma \text{ continuous}, \gamma_0=x\}.
 \end{align}
consists of all continuous paths in $M$ based at $x$. Path space is equipped with the compact-open topology.

\subsection{Wiener Measure and Stochastic Parallel Transport}\label{sec_wiener_stoch_par}

Brownian motion and stochastic parallel transport on Riemannian manifolds are most conveniently described via the Eells-Elworthy-Malliavin formalism. The gist of this construction is that Cartan's rolling without slipping provides a way to identify Brownian motion $W_t$ on $\RR^n$ with Brownian motion on $M$, as well as with horizontal Brownian motion on the frame bundle $FM$, see equation \eqref{EEM_eq} below.

To describe this, consider the $O_n$-bundle $\pi:FM\to M$ of orthonormal frames. By definition, the fiber over any point $x\in M$ is given by the orthonormal maps $u: \RR^n\to T_xM$. Thus, if $e_1,\ldots,e_n$ denotes the standard basis of $\RR^n$, then $ue_1,\ldots,ue_n$  is an orthonormal basis of $T_xM$, where $x=\pi(u)$. Recall from basic differential geometry (see e.g. \cite{KobN}) that the frame bundle comes equipped with $n$ canonical horizontal vector fields $H_1,\ldots, H_n$, which are defined by 
\begin{equation}
H_a(u)=(ue_a)^\ast,
\end{equation}
where $\ast$ denotes the horizontal lift.

Let $(P_0\RR^n,\Sigma,\P_0)$ be the space of continuous curves in $\RR^n$ equipped with the Borel $\sigma$-algebra and the Euclidean Wiener measure, and denote the coordinate process by $\bar W_t\colon P_0\RR^n\to \RR^n$. We use the normalization that the generator of $\bar W_t$ is given by $\Delta_{\RR^n}$ instead of $\frac12\Delta_{\RR^n}$, i.e. the covariation is given by
\begin{equation}\label{eq_normalization}
[\bar W_t^a,\bar W_t^b]=2t \delta_{ab}.
\end{equation}

Given an initial frame $u$ above $x\in M$, following Eells-Elworthy-Malliavin one considers the following stochastic differential equation (SDE) on the frame bundle
 \begin{align}\label{EEM_eq}
 \dd \bar{U}_t=\sum_{a=1}^nH_a(\bar{U}_t)\circ\! \dd \bar W_t^a,\quad \bar{U}_0=u,
 \end{align}
 where $H_1,\ldots, H_n$ are the canonical horizontal vector fields, and $\circ d$ denotes the Stratonovich differential.
 
 \begin{definition}[Ito map, Wiener measure, and stochastic horizontal lift]\label{def_ito_map_etc}  Let $\bar{U}:P_0\mathbb{R}^n\to P_uFM$ be the solution map of the SDE \eqref{EEM_eq}. The map $I:=\pi(\bar{U})\colon P_0\RR^n\to P_xM$ is called the Ito map. The Wiener measure on $P_xM$ is defined as the pushforward measure $\P_x=I_\#\P_0$. The map $W_t:= \bar W_t\circ I^{-1}\colon P_xM\to \mathbb R^n$ 
 is the euclidean Brownian motion under $\P_x$.
 Finally, the map $U:=\bar{U}\circ I^{-1}\colon P_xM\to P_uFM$  is called the stochastic horizontal lift.
 \end{definition}
 
By definition, the Ito map $I$ provides an isomorphism between the probability spaces $(P_0\RR^n,\Sigma,\P_0)$ and $(P_xM,\Sigma,\P_x)$. 
As stated in the introduction, the Wiener measure $\P_x$ on $P_xM$ is uniquely characterized by the following property. For any evaluation map
\begin{equation}
e_{t_1,\ldots,t_k}\colon P_xM\to M^k,\quad \gamma\mapsto (\gamma_{t_1},\ldots, \gamma_{t_k}),
\end{equation}
its pushforward is given by 
\begin{align}
  (e_{t_1,\ldots,t_k})_\#\dd \P_x(y_1,\ldots,y_k)=\rho_{t_1}(x,d y_1)\rho_{t_2-t_1}(y_1,d y_2)\cdots \rho_{t_k-t_{k-1}}(y_{k-1},d y_k),
 \end{align}
where $\rho_t(x,dy)=\rho_t(x,y)\, dv_g(y)$ denotes the heat kernel measure on $M$.
 
 The main advantage of the frame bundle formalism is that in addition to the Wiener measure of Brownian motion on $M$ it also yields (without any additional effort) a notion of stochastic parallel transport:
 
 \begin{definition}[stochastic parallel transport]\label{def_stoch_par}
 The family of isometries $P_t:=U_0U_t^{-1}\colon T_{\pi(U_t)}M\to T_xM$ is called stochastic parallel transport.
 \end{definition}
 
To conclude this section, let us point out as a consequence of the SDE \eqref{EEM_eq}, taking also into account our normalization \eqref{eq_normalization}, the Ito formula on the frame bundle takes the form
\begin{align}\label{Ito frame}
\dd \tilde f=H_a\tilde f \dd W_t^a+\Delta_H\tilde f \dd t,
\end{align}
where $\tilde f=f\circ \pi:FM\to M\to\RR$, and where $\Delta_H\equiv\sum_{a=1}^n H_a^2$ denotes the horizontal Laplacian.

\subsection{Gradients on Path Space}\label{sec_prelim_gradients}

This section is dedicated to studying the various notions of gradients which appear in this paper and relating them.  In general there are many such notions that play a role in the literature, and in this paper at some point or another, however most are easily related.  

As before, we denote by $\cH$ the Cameron-Martin space, i.e. the Hilbert space of $H^1$-curves $\{h_t\}_{t\geq0}$ in $\RR^n$ with $h_0=0$, equipped with the inner product
\begin{align}
\ip{h,k}_\cH:=\int_0^\infty\ip{\dot h_t,\dot k_t}\dd t.
\end{align}
Any $h\in\mathcal{H}$ can be viewed as a vector field $Uh$ on $P_xM$  by taking
\begin{equation}\label{h_as_vf}
(Uh)_t(\gamma)=U_t(\gamma)h_t\in T_{\gamma_t}M,
\end{equation}
where $U(\gamma)$ denotes the stochastic horizontal lift of $\gamma$ as in Definition \ref{def_ito_map_etc}.

For a function $F:P_xM\to \dR$, a priori there are several notions a priori, which can be listed:
\begin{align}
&\nabla^{\cH} F:P_xM\to \cH \text{ the Malliavin gradient},\notag\\
&\nabla^\parallel_t F:P_xM\to T_xM \text{ the parallel gradients},\notag\\
&\nabla^{\mathcal{L}} F :P_xM\to TP_xM \text{ the $L^2$-gradient}.
\end{align}

Additionally in this paper we will be considering the $\varphi$-gradient $\nabla_\varphi F:P_xM\to T_xM$.

The following summarizes the relationships between the first of these notions of gradient:

\begin{lemma}[gradients]\label{lemma_gradients}
Let $F:P_xM\to \dR$, then we have the relations:
\begin{align}
 \ip{\nabla^{\mathcal{L}} F,Uh }_{L^2} =  D_{Uh} F  = \ip{\nabla^{\cH} F,h}_{\cH} = \int_0^\infty \ip{ \nabla^\parallel_t F, \dot h_t} \, dt	\, ,
\end{align}
  	where $Uh$ denotes the vector field associated to $h\in \cH$ as in \eqref{h_as_vf}.
\end{lemma}

\begin{proof}
The first two equalities are tautological, as they are in fact the definitions of the $L^2$-gradient and the Malliavin gradient, respectively. Thus, we will focus on relating these notions of gradient to the parallel gradient as in the last equality.  It is enough for us to show this on cylinder functions
\begin{equation}
F=f(\gamma_{t_1},\ldots,\gamma_{t_k}).
\end{equation}
For a cylinder function $F$, the directional derivative $D_{Uh}F$ in direction of the vector $Uh$ is given by
\begin{align}
D_{Uh}F(\gamma)&=\sum_{j=1}^k\ip{\nabla^{(j)}f(\gamma_{t_1}, \cdots,\gamma_{t_k}),U_{t_j}(\gamma)h_{t_j} }_{T_{\gamma_{t_j}}M}\\
&=\sum_{j=1}^k\ip{ P_{t_j}(\gamma)\nabla^{(j)}f(\gamma_{t_1}, \cdots,\gamma_{t_k}),h_{t_j} }_{T_xM},
\end{align}
where $\nabla^{(j)}$ denotes the gradient with respect to the $j$-th variable, and $P_t(\gamma)$ denotes stochastic parallel transport as in Definition \ref{def_stoch_par}. Recall from \cite{Naber_char} that the $t$-parallel gradient $\nabla_t^\parallel F: P_xM\to T_xM$ is defined via the directional derivative of $F$ along the vector field which is $0$ up to time $t$ and parallel translation invariant for times larger than $t$, i.e.
\begin{equation}\label{t_par_easy}
\nabla_t^\parallel F=\sum_{t_j> t} P_{t_j}(\gamma)\nabla^{(j)}f(\gamma_{t_1}, \cdots,\gamma_{t_k}).
\end{equation}

As a motivation for the related but more complicated analysis of the Hessian in subsequent sections, it is convenient to rephrase the above in terms of the frame bundle $FM$ as in Section \ref{sec_wiener_stoch_par}. In terms of the horizontal vector fields $H^{(j)}$ on $FM^k$ we can write the directional derivative as
\begin{align}
D_{Uh} F=  \sum_{j} \ip{ H^{(j)} \tilde{f}\, ,h_{t_j}}\, ,
\end{align}
where $\tilde{f}=f\circ \pi :FM^k\to \mathbb{R}$ denotes the lift of $f$. Moreover, equation \eqref{t_par_easy} can be rewritten as
\begin{align}
\nabla^\parallel_t F = \sum_{t_j> t}	H^{(j)} \tilde{f}\, ,
\end{align}
so that
\begin{align}
\int_{t_{i-1}}^{t_{i}} \ip{ \nabla^\parallel_t F, \dot h_t}\, dt =  \sum_{j\geq i}\ip{ H^{(j)} \tilde{f}\, , h_{t_{i}}-h_{t_{i-1}} }\, .
\end{align}
Finally, let us put all of this together in order to compute
\begin{align}
D_{Uh} F&=\sum_{j} \ip{ H^{(j)} \tilde{f}\, ,h_{t_j}}=\sum_{j\geq i} \ip{ H^{(j)} \tilde{f}\, , h_{t_{i}}-h_{t_{i-1}} }=\int_0^\infty \ip{ \nabla^\parallel_t F, \dot h_t}\, dt\, ,
\end{align}
which proves the final equality.
\end{proof}

Therefore we have seen that all the notions of gradient contain roughly the same information, simply packaged in a slightly different form.  Let us use this to understand our notion of $\varphi$-gradient $\nabla_\varphi F:P_x M \to T_x M$. The following is immediate from the definition:

\begin{corollary}[gradients]
Let $F:P_xM\to \dR$ and $\varphi:[0,\infty)\to \dR$ be an $H^1$-function with $\varphi(0)=0$.  Then we have the relations
\begin{align}
\ip{ \nabla_\varphi F , v} \equiv D_{U(\varphi v)} F= \ip{\nabla^{\mathcal{L}} F,U(\varphi  v)}_{L^2} = \ip{\nabla^{\cH} F,\varphi v}_{\cH} = \int_0^\infty \ip{ \nabla^\parallel_t F, \dot \varphi  v}	\, dt\, ,
\end{align}
where we are viewing $\varphi v\in \cH$.
\end{corollary}

One can therefore view the $\varphi$-gradient as a smoothed version of the parallel gradient, where instead of defining a gradient for each $t\geq 0$ we have defined a gradient for each $\varphi:[0,\infty)\to \dR$.\\

To conclude this section, let us remark that from Lemma \ref{lemma_gradients} (gradients) one sees that
\begin{equation}
\nabla_t^\parallel F=\frac{d}{dt}(\nabla^{\cH} F)_t,
\end{equation}
i.e. the $t$-parallel gradient is the time-derivative of the Malliavin gradient. In particular, it follows that
\begin{align}
 \norm{\nabla^{\cH} F}^2_\cH=\int_0^\infty\abs{\nabla_t^\parallel F}^2\dd t.
\end{align}
Also, having defined them on cylinder functions, thanks to the integration by parts formula (see below), the gradients can be extended to unbounded closed operators on $L^2$.

\subsection{Intertwining Formula and Integration by Parts}\label{sec_int_by_parts}

Let us first recall the classical integration by parts formula on path space. If $F,G: P_xM\to\mathbb{R}$ are cylinder function and $h\in\cH$, assuming say either that $h$ is compactly supported or $\Ric=0$, then 
\begin{align}
\Ex\left[D_{Uh} F\, G \right]=\Ex\left[ -FD_{Uh}G+\frac12FG\int_0^\infty\ip{ \dot h_t+\Ric_{t}h_t, d W_t}\right],
\end{align}
see \cite{Driver_ibp, Hsu_quasi}. 
Here, $\Ric_{t}\colon \RR^n\to \RR^n$ is the Ricci transform at $U_t$, i.e. for $v\in\RR^n$, $\Ric_{t}v$ denotes the unique element in $\RR^n$ such that $\ip{\Ric_{t}v,w}=\Ric_{\pi(U_t)}(U_tv, U_tw)$ for all $w\in \RR^n$.

More generally, as pointed out e.g. in \cite[Sec. 2.3]{cruzeirofang}, instead of constant $h\in\cH$, one can also consider adapted processes $v_t: P_xM \to\mathbb{R}^n$  with $\Ex\left[ ||v||_{\cH}^2\right]<\infty$. To discuss this, recall first from Definition \ref{def_ito_map_etc} that the Ito map
\begin{equation}
I: P_0\mathbb{R}^n\to P_xM
\end{equation}
is an isomorphism between probability spaces. However, the Ito map does not preserve the geometry. The curvature term one gets from differentiating the Ito map is captured conveniently by the intertwining formula from Cruzeiro-Malliavin \cite[Thm. 2.6]{cruzeiromalliavin}: The derivative of a differentiable function $F:P_xM\to\mathbb{R}$ can be computed in terms of the derivative of the composed function $F\circ I : P_0\mathbb{R}^n\to\mathbb{R}$ via
\begin{align}\label{Bismut covariant derivative}
D_{v^\ast} (F\circ I)=(D_{Uv}F)\circ I,
\end{align}
where the $\mathbb{R}^n$-valued process $v^\ast$ is given by
\begin{align}\label{Bismut covariant derivative 2}
\dd v_t^\ast=\dd v_t-\int_0^t\cR_s(\circ d W_s,v_s) \circ d W_t,
\end{align}
where
\begin{equation}
\ip{ \cR_s(x,y)w,z}\equiv \Rm_{\gamma_s}\big(U_s(\gamma) x,U_s(\gamma) y,U_s(\gamma) z, U_s(\gamma) w\big),
\end{equation}
for $x,y,z,w\in\mathbb{R}^n$. Here, the the process $v^\ast$ is a so-called tangent process. In general, a tangent process $\zeta$ is an $\mathbb R^n$-valued semi-martingale 
\begin{align}\label{eq: tangent}
d\zeta_t=A_t \dd W_t+b_t \dd t,
\end{align}
 where $t\mapsto (A_t,b_t)$ is an adapted process taking values in $\frak{so}_n\times \mathbb R^n$ such that $\Ex\left[\int_0^\infty |b_s|^2 \dd s\right]<\infty$. The derivative of a function $\bar{F}:P_0\mathbb{R}^n\to \mathbb{R}$ in direction of a tangent process $\zeta$ is defined by
\begin{align}\label{eq: der}
D_\zeta \bar{F}(\beta)=\frac{\dd}{\dd\varepsilon}\Big|_{\eps=0} \bar{F}\left(\psi_\eps^\zeta(\beta)\right),
\end{align}
where $\beta=I^{-1}(\gamma)$ and
\begin{align}
&\psi_\eps^\zeta(\beta)_t=\int_0^te^{\eps A_s(\gamma)} \dd W_s(\gamma)+\eps\int_0^t b_s(\gamma) \dd s.
\end{align}

The intertwining formula can be used to derive the following variant of the integration by parts formula:

\begin{proposition}[{integration by parts, c.f. \cite[Sec. 2.3]{cruzeirofang}}]\label{prop_int_by_parts}
If $F,G: P_xM\to\mathbb{R}$ are cylinder functions, then for any adapted process $v_t: P_xM \to\mathbb{R}^n$ with $\Ex \left[\int_0^\infty\left(  \abs{\dot v_t}^2+\abs{\dot v_t+\Ric_{t}v_t}^2\right) \dd t\right]<\infty$, we have
\begin{align}\label{eq: ibp}
\Ex\left[D_{Uv} F\, G \right]=\Ex\left[ -FD_{Uv}G+\frac12FG\int_0^\infty\ip{ \dot v_t+\Ric_{t}v_t, d W_t}\right].
\end{align}
\end{proposition}

\begin{proof}
By the product rule it is enough to prove the integration by parts formula in the case $G=1$.

Consider the function $\bar{F}:=F\circ I: P_0\mathbb{R}^n\to \mathbb{R}$, where $I$ denotes the Ito map. Applying Girsanov's theorem, we see that on $P_0\RR^n$ for every tangent process $\zeta$ of the form \eqref{eq: tangent} we have
\begin{align}\label{adjointness}
\Eo\left[D_\zeta\bar F\right]=\frac12\Eo\left[\bar F\int_0^\infty\ip{b_t\circ I, \dd \bar W_t}\right].
\end{align}
In particular, we can apply this for $\zeta= v^\ast$ from equation \eqref{Bismut covariant derivative 2}. Using Ito calculus, we compute
\begin{align}
\int_0^t\cR_s(\circ d W_s,v_s)\circ d W_t=&\int_0^t\cR_s( \circ d W_s,v_s)\dd W_t+ \frac12 \cR_t(\circ d W_t, v_t) dW_t\\
=&\int_0^t\cR_s( \circ d W_s,v_s)\dd W_t- \Ric_t v_t dt.
\end{align}
Hence, the non-martingale part of $v^\ast$ is given by
\begin{equation}
\dot v_t\dd t+\Ric_tv_t\dd t.
\end{equation}
Thus, together with the intertwining formula \eqref{Bismut covariant derivative} we conclude that
\begin{align}
\Ex\left[D_{Uv} F\right]=\Eo\left[D_{v^\ast} \bar F\right]&=\frac12 \Eo\left[\bar F\int_0^\infty \ip{(\dot v_t+\Ric_t v_t) \circ I , \dd \bar W_t}\right]\\&=\frac12 \Ex\left[ F\int_0^\infty\ip{ \dot v_t+\Ric_t v_t , \dd W_t}\right].
\end{align}
This proves the proposition.
\end{proof}

\bigskip

\section{The Ricci-Flat Case}\label{sec: ricflat}

In this section, we prove our main theorems in the Ricci-flat case. In Section \ref{sec_halfway_rf}, we will find a certain positive quadratic form. In Section \ref{sec_PLY_rf}, we will rewrite this quadratic form in a more geometric way to prove our main differential Harnack inequality on path space (Theorem \ref{thmintro:harnack_ricci_flat}). In Section \ref{sec_PMH_rf}, we will establish the Matrix Harnack Inequality on path space. Finally, in Section \ref{sec_harnack_L2}, we express our differential Harnack inequality in terms of the $L^2$-Laplacian.

\subsection{A Positive Quadratic Form}\label{sec_halfway_rf}
The goal of this section is to prove Theorem \ref{thmRicflat: quadratic form} (Halfway Harnack).
To this end we start with the following definitions.
\begin{definition}[{adapted $L^2$-vector fields on path space}]\label{def_adapted_vf}
We denote by
\begin{equation}
L^2_\textrm{ad}(P_xM;\cH)
\end{equation}
the space of all $\Sigma_t$-adapted stochastic processes $v_t: P_xM \to \mathbb{R}^n$ with $\Ex\left[\norm{v}_{\cH}^2\right]<\infty$.
The space of adapted $L^2$-vector fields on path space $P_xM$ is defined by 
\begin{equation}
L^2_{\textrm{ad}}(P_xM;TP_xM):= \left\{ Uv  \, | \, v\in L^2_\textrm{ad}(P_xM;\cH) \right\},
\end{equation}
where
\begin{equation}
(Uv)_t(\gamma):= U_t(\gamma) v_t(\gamma)\in T_{\gamma_t}M
\end{equation}
is the vector field on path space corresponding to $v$. By definition, this gives a bijective map
\begin{equation}\label{bijection_u}
U: L^2_\textrm{ad}(P_xM;\cH) \to L^2_{\textrm{ad}}(P_xM;TP_xM).
\end{equation}
\end{definition}
Using this bijection, we can define the inner product of $V,W\in L^2_{\textrm{ad}}(P_xM;TP_xM)$ by
\begin{equation}
\ip{ V, W}_{\mathcal{H}} := \int_0^\infty \dot v_t \cdot\dot w_t \, dt,
\end{equation}
where $v=U^{-1}V$ and $w=U^{-1}W$ are the associated $\mathbb{R}^n$-valued processes.

\begin{definition}[{derivable vector field}]\label{def_differentiable_vf}
A vector field $V\in L^2_{\textrm{ad}}(P_xM;TP_xM)$ is called derivable if $D_Wv_t$ exists in $L_{\textrm{ad}}^2(P_xM,\cH)$ for all $W\in L^2_{\textrm{ad}}(P_xM;TP_xM)$, where $v=U^{-1}V$.
\end{definition}

In particular all constant vector fields, i.e. vector fields of the form $V=Uh$ for some $h\in\mathcal{H}$, are of course derivable.  The set of derivable vector fields is dense in the space of adapted vector fields.  The following is our Halfway Harnack inequality in the Ricci-flat case:

\begin{theorem}[Halfway Harnack]\label{thmRicflat: quadratic form}
Let $M$ be a Ricci-flat manifold, and let $F\colon P_xM\to\RR^+$ be a nonnegative cylinder function. Then, the quadratic form
\begin{align}
Q_F[V,V]:=\frac{\Ex[D_{V}(D_{V} F)]}{\Ex[F]}-\frac{\Ex[D_{V} F]^2}{\Ex[F]^2}+\frac{\Ex[D_{\nabla^{}_{V}V}F]}{\Ex[F]}+\frac12\frac{\Ex[F\norm{V}_\cH^2]}{\Ex[F]},
\end{align}
is nonnegative for every derivable $V\in L^2_{\textrm{ad}}(P_xM;TP_xM)$. Here, $\nabla^{}$ denotes the Markovian connection (see below).
\end{theorem}

Morally speaking, our quadratic form $Q_F$ can be thought of as ``push forward under the Ito map in the sense of adapted differential geometry" of the quadratic form $\nabla^2\Phi_F(0)$ from Remark \ref{rem_quad_form_rn}. To discuss this properly, and as a preparation for the actual proof of Theorem \ref{thmRicflat: quadratic form}, let us start by recalling the Markovian connection as introduced by Cruzeiro-Malliavin \cite{cruzeiromalliavin}.  In the following, we write\footnote{Careful about the switching of the order of $z$ and $w$ between $\cR$ and $\Rm$ below.}
\begin{equation}
\ip{ \cR_s(x,y)w,z}\equiv \Rm_{\gamma_s}\big(U_s(\gamma) x,U_s(\gamma) y,U_s(\gamma) z, U_s(\gamma) w\big),
\end{equation}
where $x,y,z,w\in\mathbb{R}^n$.

\begin{definition}[{Markovian connection, \cite[Sec. III]{cruzeiromalliavin}}]\label{def_mark_conn}
The Markovian connection is defined for constant vector fields via
 \begin{align}\label{defRicFlat:connection}
\frac{\dd}{\dd t}{U_t^{-1}({\nabla^{}_{Uk}Uh})_t}=\int_0^t\cR_s(\circ d W_s,k_s)\,\dot{h}_t,
\end{align}
with initial condition $({\nabla^{}_{Uk}Uh})_0=0$, where $h,k\in\cH$. The definition is then extended to nonconstant vector fields $Uv$ via the Leibniz rule.
\end{definition}
Note that \eqref{defRicFlat:connection} can be solved for ${\nabla^{}_{Uk}Uh}$ by integrating in time and inverting $U_t$.
The connection $\nabla^{}$ is called Markovian, since the right hand side of \eqref{defRicFlat:connection} only depends on the value of $\dot{h}$ at time $t$.

Let us observe that the Stratonovich integral actually agrees with the corresponding Ito integral for Ricci flat spaces:
\begin{lemma}\label{lemma: Stratonovich}
If $M$ is a Ricci-flat manifold, then
\begin{align}
\int_0^t\cR_s(\circ d W_s,k_s) = \int_0^t\cR_s(\dd W_s,k_s).
 \end{align}
\end{lemma}
\begin{proof}
In general, a Stratonovich integral can be converted to an Ito integral by adding a quadratic variation term:
 \begin{align}
  \int_0^tR_{abcd}^s k_s^b\circ d W_s^a=&\int_0^tR_{abcd}^s k_s^b \dd W_s^a
  +\int_0^t \tfrac12 \dd[k^bR_{abcd},W^a]_s
\end{align}
Using Ito calculus, the Bianchi identity, and the condition $\Rc=0$ we compute
\begin{align}  
\tfrac12  \dd[k^bR_{abcd},W^a]_s=k_s^bH_aR^s_{abcd}\dd s =k_s^b(H_cR^s_{bd}-H_dR_{bc}^s)\dd s=0.
 \end{align}
 This proves the lemma.
\end{proof}

It is important to note that for $V\in L^2_{\textrm{ad}}(P_xM;TP_xM)$ and derivable $W\in L^2_{\textrm{ad}}(P_xM;TP_xM)$, we have that $z_t:=U_t^{-1}({\nabla^{}_{V}W})_t$ defines an adapted $\RR^n$-valued process with finite norm $\Ex[\norm{z}_\cH^2]<\infty$. Hence, by Proposition \ref{prop_int_by_parts} (integration by parts) it holds that:
\begin{align}\label{ibp process}
 \Ex[D_{{\nabla^{}_{V}W}}F\, G]=\Ex\left[ -FD_{{\nabla^{}_{V}W}}G +\tfrac{1}{2}FG  \int_0^\infty\ip{  \dot{z}_t, d W_t}\right].
\end{align}

In general, the Markovian connection interacts well with the integration by parts formula. Recall that the integration by parts formula motivates the following definition of divergence.

\begin{definition}[divergence on Ricci flat Manifolds]
The divergence of an adapted vector field $V=Uv\in L^2_{\textrm{ad}}(P_xM;TP_xM)$ on path space of a Ricci-flat manifold is defined by
\begin{align}
\delta(V)=\tfrac12\int_0^\infty \ip{\dot {v}_t, \dd W_t}.
\end{align}
\end{definition}

The following is a very useful algebraic relation:

\begin{proposition}[{commutator formula, \cite[Thm 3.2]{cruzeirofang}}]\label{RicFlat commutation}
Assume that $M$ is Ricci-flat, and let $V,W\in L^2_{\textrm{ad}}(P_xM;TP_xM)$. If $W$ is derivable, then \begin{align}
D_{V}\delta(W)=\delta({\nabla^{}_{V}{W}})+\tfrac12\ip{V,W}_\cH.
\end{align}
\end{proposition}

\begin{proof}
Let $v=U^{-1}V$ and $w=U^{-1}W$.
By the intertwining formula \eqref{Bismut covariant derivative} differentiation on $P_xM$ along $V$ can be transformed to differentiation on $P_0\RR^n$ along $v^*$ given by $v^*_0=0$ and
\begin{align}\label{diffeq_vstar}
 \dd v^*_t=\dot v_t\dd t-\int_0^t\cR_s(\circ d W_s,v_s)\dd W_t,
\end{align}
where we replaced $\circ d W_t$ by $\dd W_t$ using the assumption $\Rc=0$, cf. the proof of Proposition \ref{prop_int_by_parts}.

Recall that curves $\beta\in P_0\RR^n$  correspond to curves $\gamma\in P_xM$ via the Ito map $I\colon P_0\RR^n\to P_xM$. The intertwining formula yields
\begin{align}\label{eq: term}
 D_{V}\int_0^\infty\ip{\dot w_t, d W_t}(\gamma)&=D_{v^*}\int_0^\infty\ip{\dot w_t\circ I, d \bar W_t}(\beta)\\
& =\int_0^\infty\ip{D_{v^*}(\dot w_t\circ I), d \bar W_t}(\beta)+\int_0^\infty\ip{\dot w_t\circ I, d D_{v^*} \bar W_t}(\beta).
\end{align}
Using again the intertwining formula, the first integrand can be rewritten as
\begin{align}\label{eq: term3}
D_{v^*}(\dot w_t\circ I)(\beta)=D_V\dot w_t(\gamma).
\end{align}
For the second term, using $\bar W_t(\beta)=\beta_t$ and equation \eqref{eq: der} we compute
\begin{align}\label{eq: term2}
 D_{v^*} \bar W_t(\beta)=&\frac{\dd }{\dd \eps}\Big|_{\eps=0}\left(\int_0^te^{-\eps\int_0^s\cR_r(\circ d W_r(\gamma),v_r(\gamma))}\dd W_s(\gamma)+\eps\int_0^t\dot { v}_s(\gamma)\dd s \right)\\
 =&-\int_0^t\int_0^s\cR_r(\circ d W_r(\gamma),v_r(\gamma))\dd W_s(\gamma)+\int_0^t\dot { v}_s(\gamma)\dd s.
\end{align}
Consequently, combining \eqref{eq: term}, \eqref{eq: term3} and \eqref{eq: term2} and, we conclude that
\begin{align}
2D_{V}\delta(W)=&\int_0^\infty\ip{D_V\dot w_t, d W_t}
-\int_0^\infty\ip{\dot{w}_t,\int_0^t\cR_r(\circ d W_r,v_r)\dd W_t}+\int_0^\infty\ip{\dot{v}_t,\dot { w}_t}\dd t\\
=&\int_0^\infty\ip{D_V\dot w_t, d W_t}+
\int_0^\infty\ip{\int_0^t\cR_r(\circ d W_r,v_r)\dot{w}_t, d W_t}+\ip{V,W}_\cH.
\end{align}
Observing that
\begin{align}
D_V\dot w_t+
\int_0^t\cR_r(\circ d W_r,v_r)\dot{w}_t=\frac{\dd}{\dd t}U_t^{-1}(\nabla_VW)_t,
\end{align}
and recalling the definition of divergence, this proves the proposition.
\end{proof}

Now we are able to check by direct computation that our quadratic form $Q_F$ is nonnegative.

\begin{proof}[{Proof of Theorem \ref{thmRicflat: quadratic form}}]
By scaling we can assume that
\begin{equation}
\Ex[F]=1.
\end{equation}
First, by the usual integration by parts formula \eqref{eq: ibp} we have
\begin{align}
{\Ex\left[D_{V}F\right]}=\Ex\left[F\delta(V)\right].
\end{align}
Second, using the version from \eqref{ibp process} we see that
\begin{align}
\Ex\left[D_{\nabla^{}_{V}V}F\right]=\Ex\left[F \delta(\nabla^{}_{V}V)\right].
\end{align}
Third, applying integration by parts twice and using Proposition \ref{RicFlat commutation} (commutator formula) we obtain
\begin{align}
\Ex\left[D_{V}(D_{V}F)\right]=&\Ex\left[D_{V}F\delta({V})\right]\\
=&\Ex\left[ F\delta({V})^2\right]-\Ex\left[F D_{V}\delta({V})\right]\\
=&\Ex\left[ F\delta({V})^2\right]-\Ex\left[F \delta(\nabla^{}_{V}{V})\right] -\tfrac12\Ex[F\norm{V}_\cH^2].
\end{align}
Combining the above formulas, we conclude that
\begin{equation}
Q_F[V,V]=\Ex\left[ F\delta({V})^2\right]-\Ex\left[ F\delta({V})\right]^2,
\end{equation}
which is indeed nonnegative by the Cauchy-Schwarz inequality. This proves the theorem.
\end{proof}

\subsection{Differential Harnack}\label{sec_PLY_rf}

We can now prove our differential Harnack inequality on path space (Theorem \ref{thmintro:harnack_ricci_flat}), which we restate here for convenience of the reader:

\begin{theorem}[Differential Harnack inequality on path space]\label{thm_harn_rf_restated}
Let $M$ be a Ricci-flat manifold, and let $F:P_xM\to \dR$ be a nonnegative function. Then, for all $\varphi\in H^1_0(\mathbb{R}^+)$ we have the inequality
\begin{align}\label{e:harnack_ricci_flat}
 &\frac{\Ex\left[\Delta_{\varphi}F\right]}{\Ex[F]}
 -\frac{\big|{\Ex\left[\nabla_\varphi F\right]}\big|^2}{\Ex[F]^2}
 +\frac{n}2\norm{\varphi}^2\geq 0.
\end{align}
\end{theorem}

\begin{proof} Let $F(\gamma)=f(\gamma_{t_1},\cdots,\gamma_{t_k})$ be a nonnegative cylinder function. By scaling we can assume that
\begin{equation}
\Ex[F]=1.
\end{equation}
By Theorem \ref{thmRicflat: quadratic form} (Halfway Harnack) and the definition of the Markovian Hessian we have
\begin{align}
Q_F[V,V]=\Ex\left[\Hess^{} F(V,V)\right]-\Ex\left[D_{V} F\right]^2+2\Ex\left[D_{\nabla^{}_{V}V}F\right]+\tfrac12\norm{V}_\cH^2\geq 0.
\end{align}
for all derivable vector fields $V\in L^2_{\textrm{ad}}(P_xM;TP_xM)$. In particular, we can apply this to $V^a$ corresponding to the process $v^a_t=\varphi_t e_a$, where $e_a\in T_xM$ is an orthonormal basis. By definition of the $\varphi$-gradient we have
\begin{equation}
D_{V^a} F=\ip{ \nabla_\varphi F,e_a}\, ,
\end{equation}
and by definition of the $\varphi$-laplacian we have 
\begin{equation}
\Delta_\varphi F = \sum_{a=1}^n\Hess^{} F(V^a,V^a).
\end{equation}

Using the formula
\begin{equation}
D_{Uw}F= \int_0^\infty \ip{\nabla_t^\parallel F, \dot{w}_t} \dd t,
\end{equation}
together with Definition \ref{def_mark_conn} (Markovian connection) and Lemma \ref{lemma: Stratonovich}, we infer that
 \begin{align}
 D_{\nabla^{}_{V^a} V^a} F&=\int_0^\infty\ip{\nabla_t^\parallel F,\int_0^t\cR_s( \dd W_s,\varphi_se_a)\dot\varphi_te_a}\dd t.
 \end{align}
Hence, summing over $a$ and using that $\Ric=0$ we conclude
\begin{equation}
\sum_{a=1}^n Q_F[V^a,V^a]= \Ex\left[\Lap_\varphi F \right]-\left| \Ex\left[\nabla_\varphi F\right]\right|^2+\frac{n}{2}\norm{\varphi}^2\geq 0.
\end{equation}
This proves the theorem.\\
\end{proof}

\subsection{Differential Matrix Harnack}\label{sec_PMH_rf}

In this section, we prove the Matrix Harnack inequality on path space in the Ricci-flat case:

\begin{theorem}[Differential Matrix Harnack inequality on path space, Ricci-flat case]
Let $M$ be a Ricci-flat manifold, and let $F:P_xM\to\RR^+$ be a nonnegative $\Sigma_T$-measurable function on path space. 
Then, for every $\varphi\in H_0^1(\mathbb{R}^+)$ we have
 \begin{align}
\frac{\Ex[\Hess^{}_{\varphi} F]}{\Ex[F]}
-\frac{\Ex[\nabla_\varphi F]\otimes \Ex[\nabla_\varphi F] }{\Ex[F]^2}&+\frac{g_x}{2}\left( 1+C_T(\Rm)\frac{\Ex[F^2]^{1/2}}{\Ex[F]}\right) \norm{\varphi}^2
\geq 0,
\end{align}
where $C_T(\Rm)<\infty$ is a constant, which converges to $0$ as $\abs{\Rm}\to0$ assuming that $T$ stays bounded.
\end{theorem}

\begin{proof}Let $F=f(\gamma_{t_1},\cdots, \gamma_{t_k})$ be a cylinder function. By scaling we can assume that
\begin{equation}
\Ex[F]=1\, \textrm{ and }\, ||\varphi||=1.
\end{equation}
Let $v\in\RR^n$ be any unit vector. Arguing similarly as in the proof of Theorem \ref{thm_harn_rf_restated}, we see that
\begin{equation}
\left(\Ex\left[\Hess^{}_{\varphi} F \right]- \Ex\left[\nabla_\varphi F\right]\otimes \Ex\left[\nabla_\varphi F\right] +\frac{g_x}{2}  \right)(v,v)\geq 
-2\Ex\left[ F\delta\left( \nabla^{}_VV \right)\right].
\end{equation}
Using Ito's isometry and the bound $|\varphi_s|\leq s^{1/2}$, we can estimate 
\begin{align}
\Ex\left[ \delta\left( \nabla^{}_VV \right)^2\right]
 \leq  \int_0^T\Ex\left[\left| \int_0^t\cR_s( \dd W_s,\varphi_sv) \right|^2\right]\abs{\dot \varphi_t}^2\dd t\leq C_T(\Rm)\, .
\end{align}
Together with the Cauchy Schwarz inequality, this implies the assertion.
\end{proof}

\bigskip

\subsection{Differential Harnack in terms of $L^2$-Laplacian}\label{sec_harnack_L2}

The goal of this section is to relate the Markovian Hessian and the $L^2$-Hessian, as needed for Corollary \ref{corintro:harnack_ricci_flat}.  The following notions of gradient of vector fields will play the dominant roles:
 \begin{align}
  &\nabla^{{\mathcal{L}}}		\text{ the $L^2$-connection}\, ,\notag\\
 &\nabla	\text{ the Markovian connection}.
 \end{align}
Here, the $L^2$-connection is the Levi-Civita connection of the $L^2$-inner product, and the Markovian connection is as in Definition \ref{def_mark_conn}.
These connections on the space of vector fields naturally induce Hessians on the space of functions by the formulas:
 \begin{align}\label{e:hessians}
 &\Hess^{\mathcal{L}} F(V,W) \equiv D_V(D_W F) - D_{\nabla^{\mathcal{L}}_V W} F\, ,\notag\\
 &	\Hess^{} F(V,W) \equiv D_V(D_W F) - D_{\nabla^{}_V W} F\, .
 \end{align}
Our goal is now to relate the two induced Hessians. Namely, we will show that
\begin{equation}\label{eq_rel_hessians}
\Hess^{} F(Uh,Uh)=\Hess^{\mathcal{L}} F(Uh,Uh) +\int_0^\infty\ip{\nabla^\parallel_t F,\cR_t(\circ d W_t,h_t)h_{t}}.
\end{equation}
To prove this, we start by expressing the $L^2$-Hessian in terms of the parallel Hessians.

 \begin{lemma}[$L^2$-Hessian and Parallel Hessian]\label{lemma_par_hess}
Let $F:P_xM\to \dR$ be a function on path space, and let $\Hess^{\mathcal{L}}  F$ its $L^2$-Hessian as defined above. Then for any $h,k\in \cH$ we have
\begin{align}
\Hess^{\mathcal{L}}  F(Uh,Uk) = \int_0^\infty\!\!\int_0^\infty\ip{\nabla_s^\parallel \nabla_t^\parallel F,\dot h_s\otimes\dot k_t}\dd s\dd t.	
\end{align}
\end{lemma}

\begin{proof}
The proof is a more involved version of the relationship between the $L^2$-gradient and the parallel gradient from the preliminaries section.  To begin, note that for any cylinder function
\begin{equation}
F(\gamma)=f(\gamma_{t_1},\ldots,\gamma_{t_k}),
\end{equation}
using the horizontal vector fields on $FM^k$, we can compute
\begin{align}
\int_{t_i}^{t_{i-1}}\!\!\int_{t_j}^{t_{j-1}}\ip{\nabla_s^\parallel \nabla_t^\parallel F,\dot h_s\otimes\dot k_t}\dd s\dd t = \sum_{\ell\geq i,m\geq j} H^{(\ell)}_aH^{(m)}_b\tilde{f}\,	( h^a_{t_{i}}-h^a_{t_{i-1}} ) ( h^b_{t_{j}}-h^b_{t_{j-1}} )\, .
\end{align}
Now, similarly as in the $L^2$-gradient computation from the preliminaries section, we can use the horizontal vector fields to compute the $L^2$-Hessian:
\begin{align}
\Hess^{\mathcal{L}} F(Uh,Uk) & = \sum_{\ell,m} 	H^{(\ell)}_{a}H^{(m)}_{b} \tilde{f}\, h^a_{t_{\ell}} k^b_{t_{m}}\notag\\
&=\sum_{\ell\geq i,m\geq j} H^{(\ell)}_aH^{(m)}_b\tilde{f}\,	( h^a_{t_{i}}-h^a_{t_{i-1}} ) ( h^b_{t_{j}}-h^b_{t_{j-1}} )\notag\\
&=\sum_{i,j}\int_{t_i}^{t_{i-1}}\!\!\int_{t_j}^{t_{j-1}}\ip{\nabla_s^\parallel \nabla_t^\parallel F,\dot h_s\otimes\dot k_t}\dd s\dd t \notag\\
&=\int_0^\infty\int_0^\infty\ip{\nabla_s^\parallel \nabla_t^\parallel F\dot h_s,\dot h_t}\dd s\dd t\, ,
\end{align}
which, by density of cylinder functions, completes the proof of the lemma.
\end{proof}

We will now prove the formula \eqref{eq_rel_hessians}, relating the $L^2$-Hessian and the Markovian-Hessian:

\begin{proposition}[Markovian Hessian and $L^2$-Hessian]\label{prop: mark hess}
Let $F\colon P_xM\to\RR$ be a cylinder function, and let $V_t=U_th_t$, where $h\in\cH$. Then
 \begin{align}\label{eq: mark hess}
\Hess^{} F(V,V)=\Hess^{\mathcal{L}}  F(V,V) +\int_0^\infty\ip{\nabla^\parallel_t F,\cR_t(\circ d W_t,h_t)h_{t}}.
 \end{align}

\end{proposition}

\begin{remark}
The integral in \eqref{eq: mark hess} is an anticipating integral, but since $F(\gamma)=f(\gamma_{t_1},\ldots,\gamma_{t_k})$ is a cylinder function it can simply be expressed as a finite sum of usual non-anticipating integrals:
\begin{align}
\int_0^\infty\ip{\nabla^\parallel_t F,\cR_t(\circ d W_t,h_t)h_{t}}
=\sum_i\ip{ u_{t_i}^{-1}\nabla^{(i)} f,\int_0^{t_i} \cR_t(\circ d W_t,h_t)h_{t}}.
\end{align}
\end{remark}

\begin{proof}
Let $F(\gamma)=f(\gamma_{t_1},\ldots,\gamma_{t_k})$ be a cylinder function, where $0<t_1<\ldots<t_k$. We will first compute in the smooth setting and appeal to the transfer principle later. So let $\gamma_t$ be a smooth curve in $M$ starting at $x$, with horizontal lift $u_t$ and anti-development $\beta_t$. Let $\gamma_t^\eps$ be a smooth variation of $\gamma_t$ with fixed initial point such that
 \begin{align}
  \frac{d}{d\eps}\Big|_{\eps=0}\gamma_t^\eps=u_t h_t=:V_t.
 \end{align}
Let $u_t^\eps$ be the horizontal lift of $\gamma_t^\eps$. We compute
\begin{align}
 D_V(D_V F)&=\frac{d}{d\eps}\Big|_{\eps=0} \sum_{i=1}^k \nabla^{(i)}_{u_{t_i}^\eps h_{t_i}} f (\gamma_{t_1}^\eps,\ldots,\gamma_{t_k}^\eps)\\
 &=\sum_{i,j=1}^k\ip{ \nabla^{(i)}\nabla^{(j)}f, V_{t_i}\otimes V_{t_j}}_{T_{\gamma_{t_i}}M\otimes T_{\gamma_{t_j}}M}+ \sum_{i=1}^k \ip{  \nabla^{(i)}f , \nabla_{V_{t_i}}V_{t_i}}_{T_{\gamma_{t_i}}M} .
\end{align}
By Lemma \ref{lemma_par_hess} ($L^2$-Hessian and parallel Hessian) the first term is given by
\begin{equation}
\sum_{i,j=1}^k\ip{ \nabla^{(i)}\nabla^{(j)}f, V_{t_i}\otimes V_{t_j}}_{T_{\gamma_{t_i}}M\otimes T_{\gamma_{t_j}}M}=\mathrm{Hess}^{\mathcal{L}} F(V,V).
\end{equation}
To compute the second term, note that by definition of the horizontal lift we have
\begin{align}
 \nabla_{\dot\gamma_{t}}(u_te_a)=0,
\end{align}
hence
\begin{align}
 \nabla_{\dot\gamma_{t}}\nabla_{V_{t}}(u_te_a)=\mathcal{R}(\dot\gamma_{t},V_{t})(u_te_a).
\end{align}
Through integration this implies
\begin{align}
 \nabla_{V_{t_i}}(u_{t_i} e_a)=&\int_0^{t_i} P_t^{t_i}(\gamma)\mathcal{R}(\dot \gamma_t,V_t)(u_te_a)\dd t\\
 =&\int_0^{t_i} u_{t_i} \cR_t(\dot \beta_t,h_t)e_a \dd t,
\end{align}
where $P_t^{t_i}=u_{t_i}\circ u_t^{-1}$ denotes the parallel transport along $\gamma$ from $T_{\gamma_t}M$ to $T_{\gamma_{t_i}}M$. Thus, we get
\begin{align}
\sum_{i=1}^k \ip{  \nabla^{(i)}f , \nabla_{V_{t_i}}V_{t_i}}_{T_{\gamma_{t_i}}M} =\sum_i\ip{ u_{t_i}^{-1}\nabla^{(i)} f,\int_0^{t_i} \cR_t(\dot{\beta}_t ,h_t) dt\, h_{t_i}}_{\mathbb{R}^n}.
\end{align}
Putting things together and using the transfer principle (see \cite{Stroock}) we obtain
\begin{align}
 D_V(D_V F)=\mathrm{Hess}^{\mathcal{L}}F(V,V)+\sum_{i=1}^k\ip{ P_{t_i}\nabla^{(i)} f,\int_0^{t_i} \cR_t(\circ d W_t,h_t)h_{t_i}}_{\RR^n}.
\end{align}
The curvature term can be rewritten as
\begin{multline}
 \sum_i\ip{ P_{t_i}\nabla^{(i)} f,\int_0^{t_i} \cR_t(\circ d W_t,h_t)h_{t_i}}_{\RR^n}\\
 =\sum_i\ip{ P_{t_i}\nabla^{(i)} f,\int_0^{t_i} \cR_t(\circ d W_t,h_t)(h_{t_i}-h_t)}_{\RR^n}
 +\sum_i\ip{ P_{t_i}\nabla^{(i)} f,\int_0^{t_i} \cR_t(\circ d W_t,h_t)h_{t}}_{\RR^n}.\label{eq: error term}
\end{multline} 
For the first term in \eqref{eq: error term} we find by recalling the definition of the Markovian connection
\begin{align}
\sum_i\ip{ P_{t_i}\nabla^{(i)} f,\int_0^{t_i} \cR_t(\circ d W_t,h_t)(h_{t_i}-h_t)}_{\RR^n}
=\sum_i\ip{ P_{t_i}\nabla^{(i)} f,\int_0^{t_i}\int_0^{t_i}1_{[t,t_i]}(s) \cR_t(\circ d W_t,h_t)\dot h_s\dd s}_{\RR^n}\\
=\sum_i\ip{ P_{t_i}\nabla^{(i)} f,\int_0^{t_i}\int_0^{s} \cR_t(\circ d W_t,h_t)\dot h_s\dd s}_{\RR^n}
=\int_0^\infty\ip{\nabla^\parallel_s F,\int_0^{s} \cR_t(\circ d W_t,h_t)\dot h_s}_{\RR^n}\dd s
=D_{\nabla^{}_VV}F,
\end{align}
where we also changed the order of integration in the first line.
For the second term in \eqref{eq: error term}, we obtain
\begin{align}
\sum_i\ip{ P_{t_i}\nabla^{(i)} f,\int_0^{t_i} \cR_t(\circ d W_t,h_t)h_{t}}_{\RR^n}
=\int_0^\infty\ip{\nabla^\parallel_t F,\cR_t(\circ d W_t,h_t)h_{t}}_{\RR^n}.
\end{align}
Putting everything together, this proves the proposition.
\end{proof}

As an immediate consequence of the above we obtain:

\begin{corollary}\label{cor_hessian}
If $M$ is Ricci-flat, then the $\varphi$-Laplacian induced by the Markovian connection and the $L^2$-connection agree, i.e.
\begin{equation}
\Lap_\varphi = \Lap_\varphi^\mathcal{L}.
\end{equation}
In particular, our differential Harnack inequality on path space of Ricci-flat manifolds can be rewritten as
\begin{align}
 &\frac{\Ex\left[\Delta^\mathcal{L}_{\varphi}F\right]}{\Ex[F]}
 -\frac{\big|{\Ex\left[\nabla_\varphi F\right]}\big|^2}{\Ex[F]^2}
 +\frac{n}2\norm{\varphi}^2\geq 0.
\end{align}
\end{corollary}

\bigskip

\section{The General Case}\label{sec: general}

Note that $\Sigma_T$-measurable functions on $C([0,\infty);M)$ can be identified with functions on $C([0,T];M)$. Hence, for ease of notation from now on we will assume that all curves have time domain $[0,1]$, i.e. we will work with the path space
\begin{equation}\label{time_restri_path_sp}
P_x M = \{ \gamma : [0,1]\to M \, | \, \gamma \textrm{ continuous }, \gamma_0=x\},
\end{equation}
the Cameron-Martin norm
\begin{equation}
|| v||_{\cH}=\left(\int_0^1 |\dot{v}_t|^2\, dt\right)^{1/2},
\end{equation}
etc (it is easy to rephrase the theorems from the introduction as equivalent theorems for $t\in[0,1]$).

\subsection{A Positive Quadratic Form}

The goal of this section is to prove Theorem \ref{thm: quadratic form} (Halfway Harnack). 
In contrast to the Ricci-flat case from the previous section, we now have to take into account the Ricci-terms. To this end, we start with the following definition.

\begin{definition}[{hat-map}]\label{hat-map}
The hat-map
\begin{equation}
L^2_\textrm{ad}(P_xM;\cH)\to L^2_\textrm{ad}(P_xM;\cH), \quad v\mapsto\hat{v}
\end{equation}
is defined by
\begin{equation}\label{eq_hat_map}
\hat{v}_t(\gamma) = v_t(\gamma)+\int_0^t \Ric_s(\gamma) v_s(\gamma)\, ds,
\end{equation}
where $\Ric_s(\gamma):\mathbb{R}^n\to\mathbb{R}^n$ is given by $\ip{\Ric_s(\gamma)v,w}=\Ric_{\gamma_s}(U_s(\gamma)v,U_s(\gamma)w)$.
\end{definition}

\begin{lemma}[{c.f. \cite[Lem. 3.7.1]{fangmall}}]\label{lemma_hat_inv}
The hat-map is well-defined, linear, and bijective. Moreover, we have the bounds
\begin{equation}
\norm{\hat{v}}_{\cH}\leq (1+C(\Ric)) \norm{v}_{\cH},\quad \textrm{and}\quad \norm{v}_{\cH}\leq (1+C(\Ric)) \norm{\hat{v}}_{\cH}
\end{equation}
where $C(\Ric)\to 0$ as $|\Ric|\to 0$.
\end{lemma}

\begin{proof}
Using $|\Ric|\leq K$ and $|{v_t}|\leq t^{1/2} ||v||_{\cH}\leq ||v||_{\cH}$ we can estimate
\begin{equation}
\int_0^1|\dot{\hat{v}}_t|^2\dd t=\int_0^1|\dot{v}_t+\Ric_t v_t|^2\dd t \leq (1+C(K)) ||v||_{\cH}^2,
\end{equation}
hence
\begin{equation}
\Ex \left[||\hat{v}||_{\cH}^2\right ]\leq (1+C(K)) \Ex \left[||{v}||_{\cH}^2\right ] <\infty.
\end{equation}
Together with the observation that by the defining formula \eqref{eq_hat_map} the process $\hat{v}_t$ is adapted whenever $v_t$ is adapted, this implies that the hat-map is well defined. Also, the hat-map is obviously linear.\\
Next, assume that $\hat{v}=0$. Then, from \eqref{eq_hat_map} we see that $v$ solves the ODE
\begin{equation}
\dot{v}_t+\Ric_t v_t = 0,\qquad v_0=0.
\end{equation}
Thus, $v=0$, which shows that the hat-map is injective.\\
Finally, given $w\in L^2_\textrm{ad}(P_xM;\cH)$ we solve the ODE
\begin{equation}
\dot{v}_t+\Ric_t v_t = \dot{w}_t, \qquad v_0=0.
\end{equation}
The solution is clearly adapted, and using $|\Ric|\leq K$ and $|{v_t}|\leq ||v||_{\cH}$ we can estimate
\begin{equation}
||v||_{\cH}^2=\int_0^1 \left|\dot{w}_t-\Ric_t v_t\right|^2\, dt \leq \left(1+{\eps}\right)||w||_{\cH}^2 + C(\eps,K) ||v||_{\cH}^2.
\end{equation}
Choosing $\eps$ small enough the term on the right hand side can be absorbed. Hence, $v\in L^2_\textrm{ad}(P_xM;\cH)$, which proves that the hat-map is surjective. This concludes the proof of the lemma.
\end{proof}

\begin{definition}[hat of a vector field on path space]\label{hat2}
For any vector field $V\in L^2_{\textrm{ad}}(P_xM;TP_xM)$ we write
\begin{equation}
\widehat{V}:=U\widehat{U^{-1}V}.
\end{equation}
\end{definition}

Now, as in Cruzeiro-Fang \cite{cruzeirofang} we can consider the modified Markovian connection:

\begin{definition}[{modified Markovian connection, \cite[Sec. 3]{cruzeirofang}}]\label{def:modifiedconnection}
The modified Markovian connection $\nabla^{\models}$ is defined via
\begin{equation}
\reallywidehat{\nabla^{\models}_{V}W}=\nabla_{V}\widehat{W}
\end{equation}
for $V,W\in L^2_{\textrm{ad}}(P_xM;TP_xM)$, where $\nabla$ denotes the Markovian connection from Definition \ref{def_mark_conn}.
\end{definition}

The modified Markovian connection is well-defined, since the hat-map is invertible by Lemma \ref{lemma_hat_inv}. Note that in the Ricci-flat case we have $\nabla^{\models}=\nabla^{}$, since the hat-map becomes the identity-map.

By \cite[Thm. 3.1]{cruzeirofang} the modified Markovian connection is compatible with the modified $\cH$-product
\begin{equation}
\ip{ V,W}_{\widehat{\cH}} := \ip{ \widehat{V},\widehat{W}}_{\cH}.
\end{equation}
Indeed, using that $\nabla$ is compatible with the $\cH$-product one can compute
\begin{align}
Z \ip{ V,W}_{\widehat{\cH}} &= Z  \ip{ \widehat{V},\widehat{W}}_{\cH} \\
& = \ip{ \nabla^{}_{Z} \widehat{V},\widehat{W}}_{\cH}+\ip{  \widehat{V},\nabla^{}_{Z} \widehat{W}}_{\cH}\\
& = \ip{ \nabla^{\models}_{Z} V,W}_{\widehat{\cH}}+\ip{  V,\nabla^{\models}_{Z} W}_{\widehat{\cH}}.
\end{align}

\begin{definition}[divergence]
The divergence of a vector field $V\in L^2_{\textrm{ad}}(P_xM;TP_xM)$ is defined by
\begin{align}
\delta(V):=\frac12\int_0^1\ip{\dot {\hat v}_t, \dd W_t},
\end{align}
where $v=U^{-1}V$ and $\dot{\hat{v}}_t=\dot{v}_t+\Ric_t v_t$ as in Definition \ref{hat-map}.
\end{definition}

The definition of the divergence is motivated by the integration by parts formula (see Section \ref{sec_int_by_parts}), which can be rewritten as 
\begin{align}\label{eq: ibp_rewritten}
\Ex\left[D_{V} F G \right]=\Ex\left[ -F D_{V}G+FG\delta(V)\right]
\end{align}
for $V\in L^2_{\textrm{ad}}(P_xM;TP_xM)$. The following is a very useful algebraic relation:

\begin{proposition}[{commutator formula, c.f. \cite[Thm. 3.2]{cruzeirofang}}]\label{commutation}
For $V,W\in L^2_{\textrm{ad}}(P_xM;TP_xM)$, with $W$ differentiable, we have
\begin{align}
D_{V}\delta(W)=\delta({\nabla^{\models}_{V}{W}})+\tfrac12\ip{ V, W}_{\widehat{\cH}}.
\end{align}
\end{proposition}

\begin{proof}
The proof is similar to the one of Proposition \ref{RicFlat commutation}, with a few changes to take into account the Ricci-terms. Generalizing equation \eqref{diffeq_vstar} we now have
\begin{align}
 \dd v^*_t=&\dot v_t\dd t-\int_0^t\cR_s(\circ d W_s,v_s)\circ d W_t\\
 =&\dot v_t\dd t-\int_0^t\cR_s(\circ d W_s,v_s)\dd W_t+\Ric_tv_t\dd t\\
 =&\dot {\hat v}_t\dd t-\int_0^t\cR_s(\circ d W_s,v_s)\dd W_t.
\end{align}
Using this and the intertwining formula we compute
\begin{align}\label{eq: term general}
2D_{V}\delta(W)&= D_{V}\int_0^\infty\ip{\dot {\hat w}_t,d W_t}\\
 &=\int_0^\infty\ip{D_{v^*}(\dot{\hat w}_t\circ I), d \bar W_t}+\int_0^\infty\ip{\dot{\hat w}_t\circ I, d D_{v^*} \bar W_t}\\
&=\int_0^\infty\ip{D_V\dot{\hat w}_t, d W_t}
-\int_0^\infty\ip{\dot{\hat w}_t,\int_0^t\cR_s(\circ d W_s,v_s)\dd W_t}+\int_0^\infty\ip{\dot{\hat v}_t,\dot {\hat w}_t}\dd t.
\end{align}
This implies the assertion.
\end{proof}

We are now ready to state and prove our Halfway Harnack inequality in the general case:

\begin{theorem}[Halfway Harnack]\label{thm: quadratic form}
Let $F:P_xM\to\mathbb{R}^+$ be a nonnegative cylinder function. Then
\begin{align}
Q_F[V,V]:=\frac{\Ex\left[D_V(D_V F)\right]}{\Ex[F]}-\frac{\Ex\left[D_{V} F\right]^2}{\Ex[F]^2}+\frac{\Ex\left[F\delta({\nabla^{\models}_{V}V})\right]}{\Ex[F]}+\frac12\frac{\Ex\left[F\norm{V}^2_{\widehat{\cH}}  \right]}{\Ex[F]}
\end{align}
is nonnegative for every derivable $V\in L^2_{\textrm{ad}}(P_xM;TP_xM)$.
\end{theorem}

\begin{proof}
By scaling we can assume that
\begin{equation}
\Ex[F]=1.
\end{equation}
Using the integration by parts formula \eqref{eq: ibp_rewritten} we get
\begin{align}
{\Ex\left[D_{V}F\right]}=\Ex\left[F\delta(V)\right],
\end{align}
and
\begin{align}
\Ex\left[D_{\nabla^{\models}_{V}V}F\right]=\Ex\left[F \delta(\nabla^{\models}_{V}V)\right].
\end{align}
Next, applying the integration by parts formula \eqref{eq: ibp_rewritten} twice and using also Proposition \ref{commutation} (commutator formula) we compute
\begin{align}
\Ex\left[D_{V}(D_{V}F)\right]=&\Ex\left[D_{V}F\delta({V})\right]\\
=&\Ex\left[ F\delta({V})^2\right]-\Ex\left[F D_{V}\delta({V})\right]\\
=&\Ex\left[ F\delta({V})^2\right]-\Ex\left[F \delta(\nabla^{\models}_{V}{V})\right] -\tfrac12\Ex\left[F\norm{V}_{\widehat{\cH}}^2\right].
\end{align}

Combining the above formulas, we conclude that
\begin{equation}
Q_F[V,V]=\Ex\left[ F\delta({V})^2\right]-\Ex\left[ F\delta({V})\right]^2,
\end{equation}
which is indeed nonnegative by the Cauchy-Schwarz inequality.
\end{proof}

\subsection{Differential Matrix Harnack}\label{sec_diff_mat_harn_gen}

In this section, we prove the differential Matrix Harnack inequality (Theorem \ref{thmintro: general}) on the path space of general Riemannian manifolds.

\begin{proof}[{Proof of Theorem \ref{thmintro: general}}]
We will show the claim for cylinder functions and appeal to density. So
let $F=f(\gamma_{t_1},\cdots,\gamma_{t_k})$ be a cylinder function.
By scaling we can assume that $F$ is $\Sigma_1$-measurable, and that
\begin{equation}\label{eq_normal}
\Ex[F]=1,\quad \textrm{and}\quad \norm{\varphi}=1.
\end{equation}
Fix any unit vector $e_a\in T_xM$. We choose
\begin{equation}\label{eq_normal2}
v_t=\varphi(t) e_a,
\end{equation}
and apply Theorem \ref{thm: quadratic form} (Halfway Harnack) for $V=Uv$, which gives
\begin{align}
{\Ex\left[D_V(D_V F)\right]}-{\Ex\left[D_{V} F\right]^2}+{\Ex\left[F\delta(\nabla^{\models}_{V}V)\right]}+\tfrac12 {\Ex\left[F\norm{V}^2_{\widehat{\cH}}  \right]}\geq 0.
\end{align}
Using the definition of the Markovian Hessian
we rewrite this as
\begin{multline}
\Ex\left[ \Hess^{} F(V,V)\right]  - \Ex\left[D_{V} F \right]^2 + \tfrac12 \Ex \left[ F\norm{V}^2_{\cH} \right] \\
+ \tfrac{1}{2}\Ex\left[F \left( \norm{V}^2_{\widehat{\cH}}-\norm{V}^2_{\cH}\right)\right]
+\Ex\left[D_{\nabla^{}_VV}  F \right]
+\Ex\left[D_{\nabla^{\models}_VV} F  \right]\geq 0.
\end{multline}
We view the terms in the second line as error terms, which we have to bound from above.

First, using Lemma \ref{lemma_hat_inv} and equations \eqref{eq_normal} and \eqref{eq_normal2} we can estimate
\begin{equation}
\tfrac{1}{2}\Ex\left[F \left( \norm{V}^2_{\widehat{\cH}}-\norm{V}^2_{\cH}\right)\right]\leq C(\Ric).
\end{equation}
Next, using also the integration by parts formula, Cauchy-Schwarz inequality, the Ito isometry, and Lemma \ref{lemma_hat_inv}, we have
\begin{equation}
\Ex\left[D_{\nabla^{}_VV}  F \right]^2=\Ex\left[ F  \delta(\nabla^{}_VV )  \right]^2 \leq  (1+C(\Ric))  \Ex[F^2]\,\Ex\left[ \norm{\nabla^{}_{V}V}_{\cH}^2  \right].
\end{equation}
Similarly, using the definition of the modified Markovian connection (Definition \ref{def:modifiedconnection}), we can estimate
\begin{equation}
\Ex\left[D_{\nabla^{\models}_VV} F  \right]^2=\Ex\left[ F \delta(\nabla^{\models}_VV )  \right]^2 \leq (1+C(\Ric)) \Ex[F^2]\, \Ex\left[ \norm{\nabla^{}_{V}\widehat{V}}_{\cH}^2  \right]  .
\end{equation}

To finish the proof of the theorem, it thus remains to prove the following claim:

\begin{claim}
We have the estimates
\begin{equation}\label{claim_eq1}
\Ex\left[ \norm{\nabla^{}_{V}V}_{\cH}^2  \right]\leq C(\Rm,\nabla\Ric),
\end{equation}
and
\begin{equation}\label{claim_eq2}
\Ex\left[ \norm{\nabla^{}_{V}\widehat{V}}_{\cH}^2  \right]\leq C(\Rm,\nabla\Ric),
\end{equation}
where $C(\Rm,\nabla\Ric)<\infty$ is a constant which tends to zero as $|\Rm|+|\nabla \Ric|\to 0$.
\end{claim}

\begin{proof}[Proof of the claim]
By definition of the Markovian connection and our choice of $V$ we have
\begin{align}
\Ex\left[ \norm{\nabla^{}_{V}V}_{\cH}^2  \right]=\Ex\left[ \int_0^1 \left| \int_0^t \mathcal{R}_s(\circ dW_s,v_s)\dot{v}_t \right|^2 \, dt\right]\leq \sup_{t\in[0,1]} \Ex\left[\left|\int_0^t  \mathcal{R}_s(\circ dW_s,v_s) \right|^2 \, \right].
\end{align}
Using Ito's lemma and the Bianchi identity we see that
\begin{equation}
\mathcal{R}_s(\circ d W_s,v_s)=\mathcal{R}_s(\dd W_s,v_s)+(\nabla\Ric)_s \wedge v_s \dd s,
\end{equation}
where $\wedge$ is a certain bilinear pairing whose precise structure is irrelevant for our purpose. Hence, using also the bound $|v_s|\leq 1$, and Ito's isometry, we can estimate
\begin{equation}
\begin{aligned}\label{eq: Ito}
\Ex\left[\left|\int_0^t  \mathcal{R}_s(\circ dW_s,v_s) \right|^2 \, \right]
&\leq 2 \Ex\left[ \left| \int_0^t\mathcal{R}_s( dW_s,v_s) \right|^2 \, \right]+2 \Ex\left[\left(  \int_0^t  |(\nabla\Ric)_s \wedge v_s| \, ds \right)^2 \, \right]\\
&\leq C(\Rm,\nabla \Ric),
\end{aligned}
\end{equation}
which proves the estimate \eqref{claim_eq1}.

Concerning estimate \eqref{claim_eq2},
by the definition of the Markovian connection  we have
\begin{align}
\Ex\left[ \norm{\nabla^{}_{V}\widehat V}_{\cH}^2  \right]=&\Ex\left[ \int_0^1 \left| D_V\dot{\hat v}_t+\int_0^t \mathcal{R}_s(\circ d W_s,v_s)\dot{\hat v}_t \right|^2 \dd t\right]\\
\leq  &2\Ex\left[\int_0^1 \abs{D_V\dot{\hat v}_t}^2 \dd t \right]+2\Ex\left[\int_0^1\left|\int_0^t  \mathcal{R}_s(\circ d W_s,v_s) \right|^2\abs{\dot{\hat v}_t}^2 \dd t \right]\,.
\end{align}
Using Lemma \ref{lemma: derivative} below, and $|v_t|\leq 1$, we can estimate
\begin{equation}
\abs{D_V\dot{\hat v}_t}^2\leq 2 |(\nabla\Ric)_t|^2+8 |\Ric_t|^2\left|\int_0^t\cR_s(\circ d W_s,v_s)\right|^2.
\end{equation}
This yields
\begin{equation}
\Ex\left[\int_0^1 \abs{D_V\dot{\hat v}_t}^2 \dd t \right]\leq C(\Rm,\nabla\Ric).
\end{equation}
Finally, using $\abs{\dot{\hat v}_t}^2\leq 2\abs{\dot v_t}^2+2\abs{\Ric}^2$ and arguing
similarly as above we can estimate
\begin{align}
\Ex\left[\int_0^1\left|\int_0^t  \mathcal{R}_s(\circ d W_s,v_s) \right|^2 \abs{\dot{\hat v}_t}^2\dd t \right]
\leq C(\Rm, \nabla\Ric).
\end{align}
This proves the claim.
\end{proof}

Putting things together we conclude that
\begin{equation}
\Ex\left[ \Hess^{} F(V,V)\right]  - \Ex\left[D_{V} F \right]^2 +\tfrac12 \Ex \left[ F\norm{V}^2_{\cH} \right]+C(\Ric)
+C(\Rm,\nabla\Ric) \Ex [F^2 ]^{1/2}\geq 0.
\end{equation}
Together with the definition of the $\varphi$-Hessian and $\varphi$-Laplacian, and our choice of $V$, this finishes the proof of Theorem \ref{thmintro: general}.
\end{proof}

It remains to prove the following lemma, which has been used in the above proof:

\begin{lemma}\label{lemma: derivative}
If $v\in\cH$ and $V=Uv$, then
 \begin{align}
  D_{V}\dot{\hat v}_t=  (\nabla\Ric)_t(v_t,v_t)+\Ric_{t}\left(\int_0^t\cR_s(\circ d W_s,v_s)v_t\right)-\int_0^t\cR_s(\circ d W_s,v_s)\Ric_tv_t.
 \end{align}
\end{lemma}

\begin{remark}
Note that in the Einstein case $D_{V}\dot{\hat v}_t=0$, as expected.
\end{remark}
 \begin{proof}
 By the definition of $\hat v$ we have
 \begin{align}
  D_V\dot{\hat v}_t=D_V\Ric_t v_t.
 \end{align}

Let us assume that $\gamma_t$ is a smooth path in $M$, and let $\gamma_t^\eps$ be the variation with $\gamma_0^\eps=0$ and $\frac{d}{d\eps}\Big|_{\eps=0}\gamma_t^\eps=V_t$. Let $u_t^\eps$ be the horizontal lift of $\gamma_t^\eps$. Let $\beta$ be the anti-development in $\RR^n$. Later we will appeal to the transfer principle.

Let $e_a$ be a basis vector in $\mathbb R^n$. Then
\begin{align}
\ip{D_V\Ric_t(v_t),e_a}_{\mathbb R^n}=&\frac{d}{d\eps}\Big|_{\eps=0}\ip{\Ric_{\gamma_t^\eps}(u_t^\eps v_t),u_t^\eps e_a}_{T_{\gamma_t^\eps}M}\\
=&\ip{\nabla_{V_t}\left(\Ric_{\gamma_t}(V_t)\right),u_t e_a}_{T_{\gamma_t}M}
+\ip{\Ric_{\gamma_t}(V_t),\nabla_{V_t}(u_te_a)}_{T_{\gamma_t}M}.
\end{align}
From the proof of Proposition \ref{prop: mark hess} we already know that
\begin{align}
\nabla_{V_t}(u_te_a)=u_t\int_0^t\cR^s(\dot \beta_s,v_s)\dd s\, e_a\, .
\end{align}
Using also the Leibniz rule we obtain
\begin{align}
\nabla_{V_t}\left(\Ric_{\gamma_t}(V_t)\right)&=(\nabla_{V_t}\Ric_{\gamma_t})(V_t)+\Ric_{\gamma_t}(\nabla_{V_t}V_t)\\
&=(\nabla_{V_t}\Ric_{\gamma_t})(V_t)+\Ric_{\gamma_t}\left( u_t\int_0^t\cR^s(\dot \beta_s,v_s)\dd s\, v_t\right) \, .
\end{align}
Putting things together, this yields
\begin{align}
D_V\Ric_tv_t=(\nabla\Ric)_t(v_t,v_t)+\Ric_{t}\left(\int_0^t\cR^s(\dot\beta_s,v_s)\dd s\, v_t\right)-\int_0^t\cR^s(\dot\beta_s,v_s)\dd s\Ric_tv_t\, .
\end{align}
By the transfer principle, this implies the assertion.
\end{proof}

\subsection{Differential Harnack}

In this final section, we prove the differential Harnack inequality on path space of general manifolds (Theorem \ref{theorem_harn_gen}) and its corollary (Corollary \ref{cor_matr_harn}).

We note that taking the trace of the differential Matrix Harnack inequality (Theorem \ref{thmintro: general}) one immediately obtains
 \begin{align}\label{eq_almost_harnack}
\frac{\Ex[\Delta^{}_{\varphi} F]}{\Ex[F]}
-\frac{\abs{\Ex[\nabla_\varphi F]}^2}{\Ex[F]^2}
+\left( \frac{n}{2}+C(\Ric)
+C(\Rm,\nabla\Ric)\frac{\Ex[F^2]^{1/2} }{\Ex[F]}  \right) 
\norm{\varphi}^2\geq 0,
\end{align}
however, only with the information that $C(\Rm,\nabla\Ric)\to 0$ as $|\Rm|+|\nabla \Ric|\to 0$.

To get the sharper estimate from Theorem \ref{theorem_harn_gen}, where  $C(\Rm,\nabla \Ric)$ tends to zero as $|\Rc|+|\nabla \Ric|\to 0$ assuming only that $|\Rm|$ stays bounded, we will argue in the opposite order. Namely, will first take the trace, and then derive sharper estimates for the error terms of the trace Harnack.

\begin{proof}[Proof of Theorem \ref{theorem_harn_gen}]
By scaling we can assume that $F$ is $\Sigma_1$-measurable, and that
\begin{equation}\label{eq_normaliz}
\Ex[F]=1,\quad \textrm{and}\quad \norm{\varphi}=1.
\end{equation}
Arguing similarly as in the proof of Theorem \ref{thmintro: general} and taking the trace over $V_a=U(\varphi e_a)$, where $e_a\in T_xM$ is an orthonormal basis, we obtain
\begin{multline}
\Ex\left[ \Delta^{}_{\varphi} F \right]  - \left|\Ex\left[\nabla_\varphi F \right]\right|^2 +\frac{n}{2} \Ex \left[ F\right] \\
+C(\Ric)+\left(1+C(\Ric)\right)\left(      \Ex\left[ \norm{\sum_{a}\nabla^{}_{V_a}V_a}_{\cH}^2  \right]^{1/2}  +\Ex\left[ \norm{\sum_{a}\nabla^{}_{V_a}\widehat{V_a}}_{\cH}^2  \right]^{1/2}       \right) \Ex[ F^2 ]^{1/2}\geq 0.
\end{multline}

To finish the proof of the theorem, it thus remains to prove the following claim:

\begin{claim}
We have the estimates
\begin{equation}\label{claim_eq1tr}
\Ex\left[ \norm{\sum_{a}\nabla^{}_{V_a}V_a}_{\cH}^2  \right]\leq C_1(\Ric,\nabla R),
\end{equation}
and
\begin{equation}\label{claim_eq2tr}
\Ex\left[ \norm{\sum_{a}\nabla^{}_{V_a}\widehat{V_a}}_{\cH}^2  \right]\leq C_2(\Rm,\nabla \Ric),
\end{equation}
where $C_1(\Ric,\nabla R)$ tends to zero as $|\Rc|+|\nabla R|\to 0$, and $C_2(\Rm,\nabla \Ric)$ tends to zero as $|\Rc|+|\nabla \Ric|\to 0$ assuming only that $|\Rm|$ stays bounded.
\end{claim}

\begin{proof}[Proof of the claim]
Using the definition of the Markovian connection and our choice of $V_a$ we have
\begin{align}
\Ex\left[ \norm{\sum_{a}\nabla^{}_{V_a}V_a}_{\cH}^2  \right]&=\Ex\left[ \int_0^1 \left| \int_0^t \sum_a\mathcal{R}_s(\circ dW_s,\varphi_se_a)\dot{\varphi}_te_a \right|^2 \, dt\right]\\
&\leq \sup_{t\in[0,1]} \Ex\left[ \left| \int_0^t \varphi_s\Ric_s \circ dW_s \right|^2 \, \right].
\end{align}
Using Ito's lemma and the contracted Bianchi identity we see that
\begin{equation}
\Ric_s \circ dW_s=\Ric_s dW_s+\tfrac{1}{2}(\nabla R)_s\, ds.
\end{equation}
Hence, using also the bound $|v_s|\leq 1$, and Ito's isometry, we can estimate
\begin{align}
\Ex\left[ \left| \int_0^t \varphi_s\Ric_s \circ dW_s \right|^2 \, \right]\leq C(\Ric,\nabla R),
\end{align}
which proves the estimate \eqref{claim_eq1tr}.

Concerning estimate \eqref{claim_eq2tr},
by the definition of the Markovian connection  we have
\begin{align}
\Ex\left[ \norm{\sum_a\nabla^{}_{V_a}\widehat V_a}_{\cH}^2  \right]=&\Ex\left[ \int_0^1\left|  \sum_a D_{V_a}\dot{\hat v}^a_t+\int_0^t \sum_a \mathcal{R}_s(\circ d W_s,v^a_s)\dot{\hat v}^a_t \right|^2 \dd t\right].
\end{align}
Now, similarly as in the proof of \eqref{claim_eq2} we can estimate
\begin{equation}
\Ex\left[ \int_0^1\left|   D_{V_a}\dot{\hat v}^a_t\right|^2\right]\leq C(\Rm,\nabla\Ric),
\end{equation}
where $C(\Rm,\nabla \Ric)$ tends to zero as $|\Rc|+|\nabla \Ric|\to 0$ assuming only that $|\Rm|$ stays bounded.
Moreover, since $\dot{\hat{v}}_t^a=\dot{v}_t^a+\Rc_t v_t^a$ and $v_t^a=\varphi_t e_a$ we have
\begin{equation}
\sum_a \mathcal{R}_s(\circ d W_s,v^a_s)\dot{\hat v}^a_t=\int_0^t\varphi_s\Rc_s \circ d W_s\,\dot{\varphi}_t+ \int_0^t \mathcal{R}_s(\circ d W_s,v_s^a) \Ric_t v_t^a
\end{equation}
From this, the assertion follows.
\end{proof}

Using the claim, and putting things together we concluded that
 \begin{equation}
{\Ex[\Delta^{}_{\varphi} F]}
-{\abs{\Ex[\nabla_\varphi F]}^2}
+ \frac{n}{2}+C(\Ric)+C(\Rm,\Rc,\nabla\Ric){\Ex\left[ F^2\right]^{1/2} } \geq 0,
\end{equation}
where $C(\Rm,\nabla \Ric)$ tends to zero as $|\Rc|+|\nabla \Ric|\to 0$ assuming only that $|\Rm|$ stays bounded. This finishes the  proof of Theorem \ref{theorem_harn_gen}.
\end{proof}

\begin{proof}[Proof of Corollary \ref{cor_matr_harn}]
Inspecting the above proof we see that in the Einstein case the error estimates in the claim above only depend on the Einstein constant. This proves the corollary.
\end{proof}

\bibliography{HKN_DiffHarnack}

\bibliographystyle{alpha}

\vspace{10mm}
{\sc Robert Haslhofer, Department of Mathematics, University of Toronto, 40 St George Street, Toronto, ON M5S 2E4, Canada}\\

{\sc Eva Kopfer, Institut f\"ur angewandte Mathematik, Universit\"at Bonn, Endenicher Allee 60, 53115 Bonn, Germany}\\

{\sc Aaron Naber, Department of Mathematics, Northwestern University, 2033 Sheridan Road, Evanston, IL 60208, USA}\\

\emph{E-mail:} roberth@math.toronto.edu, eva.kopfer@iam.uni-bonn.de, anaber@math.northwestern.edu
\end{document}